\documentclass[a4paper,11pt]{amsart}
\usepackage{amsfonts,amssymb,amsmath,amsthm,abstract}
\usepackage[ps,all,arc,rotate]{xy}
\usepackage[lmargin=1in,rmargin=1in,tmargin=1in,bmargin=1in]{geometry}
\usepackage{fancyhdr}

\newtheorem{prop}{Proposition}[section]
\newtheorem{lemma}[prop]{Lemma}
\newtheorem{thm}[prop]{Theorem}
\newtheorem{cor}[prop]{Corollary}

\theoremstyle{definition}
\newtheorem{defn}[prop]{Definition}

\newtheorem{rmk}[prop]{Remark}

\newtheorem{ass}[prop]{Assumption}

\mathchardef\mhyphen="2D
\DeclareMathOperator{\rk}{rk}
\DeclareMathOperator{\im}{Im}

\newcommand{\ra}{\rightarrow} 

\def\cA{\mathcal A}\def\cB{\mathcal B}
\def\cE{\mathcal E}\def\cF{\mathcal F}\def\cG{\mathcal G}\def\cH{\mathcal H}
\def\cI{\mathcal I}\def\cL{\mathcal L}
\def\cO{\mathcal O}
\def\cR{\mathcal R}
\def\cU{\mathcal U}

\def\CC{\mathbb C}

\def\PP{\mathbb P}
\def\QQ{\mathbb Q}

\def\ZZ{\mathbb Z}

\def\fT{\mathfrak{T}} \def\fD{\mathfrak{D}} \def\fB{\mathfrak{B}}
\def\cxE{\mathcal{E}^\cdot} \def\cxF{\mathcal{F}^\cdot} \def\cxG{\mathcal{G}^\cdot}
\def\uP{\underline{P}} \def\us{\underline{\sigma}} \def\uc{\underline{\chi}} \def\ue{\underline{\eta}}
\def\GL{\mathrm{GL}} \def\SL{\mathrm{SL}} \def\d{\cdots}
\def\Sb{\mathrm{Stab} \: \beta}

\title{Stratifications of parameter spaces for complexes by cohomology types}

\author{Victoria Hoskins}

\begin{document}

\maketitle

\renewcommand{\abstractnamefont}{\scshape}

\begin{abstract}
We study a collection of stability conditions (in the sense of Schmitt) for complexes of sheaves over a smooth complex projective variety indexed by a positive rational parameter. We show that the Harder--Narasimhan filtration of a complex for small values of this parameter encodes the Harder--Narasimhan filtrations of the cohomology sheaves of this complex. Finally we relate a stratification into locally closed subschemes of a parameter space for complexes associated to these stability parameters with the stratification by Harder--Narasimhan types.
\end{abstract}

\section{Introduction}

Let $X$ be a smooth complex projective variety and $\cO_X(1)$ be an ample invertible sheaf on $X$. We consider the moduli of (isomorphism classes of) complexes of sheaves on $X$, or equivalently moduli of $Q$-sheaves over $X$ where $Q$ is the quiver
\[ \bullet \ra \bullet \ra \cdots \cdots \ra \bullet \ra \bullet \]
with relations imposed to ensure the boundary maps square to zero. Moduli of quiver sheaves have been studied in \cite{ac,acgp,gothen_king,schmitt05}. There is a construction of moduli spaces of S-equivalence classes of \lq semistable' complexes due to Schmitt \cite{schmitt05} as a geometric invariant theory quotient of a reductive group $G$ acting on a parameter space $\fT$ for complexes with fixed invariants.  The notion of semistability is determined by a choice of stability parameters and %, as with many notions of semistability, every complexes of torsion free sheaves has a Harder--Narasimhan filtration with respect to these stability parameters. This notion of semistability 
the motivation comes from physics; it is closely related to a notion of semistability coming from a Hitchin--Kobayashi correspondence for quiver bundles due to \'{A}lvarez-C\'onsul and Garc\'ia-Prada \cite{acgp}. The stability parameters are also used to determine a linearisation of the action. The notion of S-equivalence is weaker than isomorphism and arises from the GIT construction of these moduli spaces which results in some orbits being collapsed. 

As the notion of stability depends on a choice of parameters, we can ask if certain parameters reveal information about the cohomology sheaves of a complex. We show that there is a collection of stability parameters which can be used to study the cohomology sheaves of a complex. Analogously to the case of sheaves, every unstable complex has a unique maximally destabilising filtration known as its Harder--Narasimhan filtration. In this paper we give a collection of stability parameters indexed by a rational parameter $\epsilon >0$ and show for all sufficiently small values of $\epsilon$ the Harder--Narasimhan filtration of a given complex with respect to these parameters encodes the Harder--Narasimhan filtrations of the cohomology sheaves in this complex. We then go on to study a stratification of the parameter space $\fT$ associated to these stability parameters.

Given an action of a reductive group $G$ on a projective scheme $B$ with respect to an ample linearisation $\cL$, there is an associated stratification $\{ S_\beta : \beta \in \cB \}$ of $B$ into $G$-invariant locally closed subschemes for which the open stratum is the geometric invariant theory (GIT) semistable set $B^{ss}$ \cite{hesselink,kempf_ness,kirwan}. When $B$ is a smooth variety, this stratification comes from a Morse type stratification associated to the norm square of the moment map for this action. This stratification has a completely algebraic description which can be extended to the above situation of a linearised action on a projective scheme (cf. \cite{hoskinskirwan} $\S$4). We give a brief summary of the algebraic description given in \cite{kirwan} as this will be used later on.

If we choose a compact maximal torus $T$ of $G$ and positive Weyl chamber $\mathfrak{t}_+$ in the Lie algebra $\mathfrak{t}$ of $T$, then the index set $\cB$ can be identified with a finite set of rational weights in $\mathfrak{t}_+$ as follows. By fixing an invariant inner product on the Lie algebra $\mathfrak{K}$ of the maximal compact subgroup $K$ of $G$, we can identify characters and cocharacters as well as weights and coweights. There are a finite number of weights for the action of $T$ on $B$ and the index set $\mathcal{B}$ can be identified with the set of rational weights in $\mathfrak{t}_+$ which are the closest points to $0$ of the convex hull of a subset of these weights.

We say a map $\lambda : \CC^* \ra G$ (which is not necessarily a group homomorphism) is a rational one-parameter subgroup if $\lambda( \CC^*)$ is a subgroup of $G$ and there is a integer $N$ such that $\lambda^N$ is a one-parameter subgroup (1-PS) of $G$. Associated to $\beta $ there is a parabolic subgroup $P_\beta \subset G$, a rational 1-PS $\lambda_\beta : \CC^* \ra T_{\CC}$ and a rational character $\chi_\beta : T_{\CC} \ra \CC^*$ which extends to a character of $P_\beta$. 
Let $Z_\beta$ be the components of the fixed point locus of $\lambda_\beta$ acting on $B$ on which $\lambda_\beta$ acts with weight $|| \beta ||^2$ and $Z_\beta^{ss}$ be the GIT semistable subscheme for the action of the reductive part $\Sb$ of $P_\beta$ on $Z_\beta$ with respect to the linearisation $\mathcal{L}^{\chi_{-\beta}}$ (which is the original linearisation $\cL$ twisted by the character $\chi_{-\beta} : \Sb \ra \CC^*$). Then $Y_\beta$ (resp. $Y_\beta^{ss}$) is defined to be the subscheme of $B$ consisting of points whose limit under the action of $\lambda_\beta(t)$ as $t \to 0$ lies in $Z_\beta$ (resp. $Z_\beta^{ss}$). There is a retraction $p_\beta : Y_\beta \ra Z_\beta$ given by taking a point to its limit under $\lambda_\beta$. 
%For $\beta \neq 0$ we have, $S_\beta = G Y_\beta^{ss} \cong G \times^{P_\beta} Y_\beta^{ss}$ by \cite{kirwan}. 
By \cite{kirwan}, for $\beta \neq 0 $ we have
\[ S_\beta = G Y_\beta^{ss} \cong G \times^{P_\beta} Y_\beta^{ss}. \]
The definition of $S_\beta$ makes sense for any rational weight $\beta$, although $S_\beta$ is nonempty if and only if $\beta$ is an index.

This stratification has a description in terms of Kempf's notion of adapted 1-PSs due to Hesselink \cite{hesselink}. Recall that the Hilbert--Mumford criterion states a point $b \in B$ is semistable if and only if it is semistable for every 1-PS $\lambda$ of $G$; that is, $\mu^{\cL}(b, \lambda) \geq 0$ where $\mu^{\cL}(b,\lambda) $ is equal to minus the weight of the $\CC^*$-action induced by $\lambda$ on the fibre of $\cL$ over $\lim_{t \to 0} \lambda(t) \cdot b$. In \cite{kempf} Kempf defines a non-divisible 1-PS to be adapted to an unstable point $b \in B - B^{ss}$ if it minimises the normalised Hilbert--Mumford function:
\[ \mu^{\cL}(b, \lambda) = \min_{\lambda'} \frac{\mu^{\cL}(b, \lambda')}{|| \lambda'||}. \]
Hesselink used these adapted 1-PSs to stratify the unstable locus and this stratification agrees with the stratification described above. In fact if $\beta$ is an nonzero index, then the associated 1-PS $\lambda_\beta$ is a 1-PS which is adapted to every point in $Y_\beta^{ss}$.

In this paper we study the stratification obtained in this way from a suitable action of a group $G$ on a parameter space for complexes using the above collection of stability parameters (for very small $\epsilon$) which are related to cohomology. We show that for a given Harder--Narasimhan type $\tau$, the set up of the parameter scheme can be chosen so all sheaves with Harder--Narasimhan type $\tau$ are parametrised by a locally closed subscheme $R_\tau$ of the parameter space $\fT$. Moreover, $R_\tau$ is a union of connected components of a stratum $S_{\beta(\tau)}$ in the associated stratification. The scheme $R_\tau$ has the nice property that it parametrises complexes whose cohomology sheaves are of a fixed Harder--Narasimhan type. % In particular this allows us to stratify the parameter space for complexes in such a way that in each stratum the cohomology sheaves have a fixed Harder--Narasimhan type.

%If there is a quotient of the $G$-action on $R_\tau$, then it would have the nice property that for two complexes to be S-equivalent it is necessary that their cohomology sheaves have the same Harder--Narasimhan type. Following \cite{hoskinskirwan}, there is a categorical quotient of the $G$-action on $R_\tau$ which we also describe in this paper. Unfortunately the categorical quotient collapses too many orbits and so does not provide a suitable quotient. We end by discussing how to provide a suitable quotient by using techniques similar to those given in \cite{hoskinskirwan}.

The layout of this paper is as follows. In $\S$\ref{schmitt construction} we give a summary of the construction of Schmitt of moduli spaces of complexes and study the action of 1-PSs of $G$. In $\S$\ref{sec on stab} we give the collection of stability conditions indexed by $\epsilon >0$ and show that the Harder--Narasimhan filtration of a complex (for very small $\epsilon$) encodes the Harder--Narasimhan filtration of the cohomology sheaves. Then in $\S$\ref{sec on strat} we study the associated GIT stratification of the parameter space for complexes and relate this to the stratification by Harder--Narasimhan types. Finally, in $\S$\ref{sec on quot} we consider the problem of taking a quotient of the $G$-action on a Harder--Narasimhan stratum $R_\tau$.

\subsection*{Notation and conventions}
Throughout we let $X$ be a smooth complex projective variety and $\cO_X(1)$ be an ample invertible sheaf on $X$. All Hilbert polynomials of sheaves over $X$ will be calculated with respect to $\cO_X(1)$. 
We use the term complex to mean a bounded cochain complex of torsion free sheaves. We say a complex $\cxE$ is concentrated in $[m_1,m_2]$ if $\cE^i = 0$ for $i < m_1 $ and $i>m_2$.

\subsection*{Acknowledgements}
I am very grateful to my thesis supervisor Frances Kirwan for her support and guidance over the last few years.

\section{Schmitt's construction}\label{schmitt construction}

In this section we give a summary of the construction due to Schmitt \cite{schmitt05} of moduli space of S-equivalence classes of semistable complexes over $X$. We also make some important calculations about the weights of $\CC^*$-actions (some details of which can also be found in \cite{schmitt05} Section 2.1). 

If we have an isomorphism between two complexes $\cxE$ and $\cxF$, then for each $i$ we have an isomorphism between the sheaves $\cE^i$ and $\cF^i$ and thus an equality of Hilbert polynomials $P(\cE^i) = P(\cF^i)$. Therefore we can fix a collection of Hilbert polynomials $P = (P^i)_{i \in \ZZ}$ such that $P^i=0$ for all but finitely many $i$ and study complexes with these invariants. In fact we can assume $P$ is concentrated in $[m_1,m_2]$ and write $P = (P^{m_1},\d,P^{m_2})$.

\subsection{Semistability}\label{stab defn}

The moduli spaces of complexes only parametrise a certain collection of complexes with invariants $P$ and this collection is determined by a notion of (semi)stability. Schmitt introduces a notion of (semi)stability for complexes which depends on a collection of stability parameters $(\us, \uc)$ where $\uc := \delta \ue$ and
\begin{itemize}
\item $\us=(\sigma_i \in \ZZ_{>0})_{i \in\ZZ}$,
\item $\ue=(\eta_i \in \QQ)_{i \in\ZZ}$,
\item $\delta$ is a positive rational polynomial such that $\deg \delta =\max(\dim X-1,0)$.
\end{itemize}
%and $\uc := \delta \ue$. 

\begin{defn} %Let $\mathcal{E}_{\cdot}$ be a complex which is concentrated in degrees $1$ to $m$. 
The reduced Hilbert polynomial of a complex $\cxF$ with respect to the parameters $(\us, \uc)$ is defined as
\[P_{\us, \uc}^{\mathrm{red}}(\cxF):= \frac{\sum_{i \in \ZZ} \sigma_i P(\cF^i) - \chi_i \rk{\cF^i}}{\sum_{i \in \ZZ} \sigma_i \rk \cF^i} \]
where $P(\cF^i)$ and $ \rk{\cF^i}$ are the Hilbert polynomial and rank of the sheaf $\cF^i$. %We call this weighted analogue of the reduced Hilbert polynomial the {reduced Hilbert polynomial with respect to $(\us, \uc)$} or simply the reduced Hilbert polynomial if the choice of stability parameters is clear. 

We say a nonzero complex $\cxF$ is $(\us, \uc)$-semistable if for any nonzero proper subcomplex $\cxE \subset \cxF$ we have an inequality of polynomials
\[P_{\us,\uc}^{\mathrm{red}}(\cxE) \leq P_{\us, \uc}^{\mathrm{red}}(\cxF). \]
By an inequality of polynomials $R \leq Q$ we mean $R(x) \leq Q(x)$ for all $x >\!> 0$. We say the complex is $(\us, \uc)$-stable if this inequality is strict for all such subcomplexes.
\end{defn}

\begin{rmk}\label{normalise} Observe that for any rational number $C$ if we let $\ue' = \ue - C\us$ and $\uc' = \delta \ue'$, then the notions of $(\us, \uc)$-semistability and $(\us, \uc')$-semistability are equivalent. For invariants $P=(P^{m_1},\d,P^{m_2})$ and any stability parameters $(\us,\uc)$, we can let \[C = \frac{\sum_{i=m_1}^{m_2} \eta_i r^i}{\sum_{i=m_1}^{m_2} \sigma_i r^i}\] where $r^i$ is the rank determined by the leading coefficient of $P^i$ and consider the associated stability parameters $(\us,\uc')$ for $P$ which satisfy $\sum_{i=m_1}^{m_2} \eta_i' r^i = 0$. As we have fixed $P$ in this section, we may assume we have stability parameters which satisfy $\sum_{i=m_1}^{m_2} \eta_i r^i = 0$.
\end{rmk}

\subsection{The parameter space}\label{param sch const}

The set of sheaves occurring in a $(\us,\uc)$-semistable complex $\cxE$ with invariants $P$ is bounded by the usual arguments (see \cite{simpson}, Theorem 1.1) and so we may choose $n>\!>0$ so that all these sheaves are $n$-regular. 
Fix complex vector spaces $V^i$ of dimension $P^i(n)$ and let $Q^i$ be the open subscheme of the quot scheme $\mathrm{Quot}(V^i \otimes \cO_X(-n), P^i)$ consisting of torsion free quotient sheaves $q^i: V^i \otimes \cO_X(-n) \rightarrow \cE^i$ such that $H^0(q^i(n))$ is an isomorphism. The parameter scheme $\fT$  for $(\us,\uc)$-semistable complexes with invariants $P$ is constructed as a locally closed subscheme of a projective bundle $\fD$ over the product $Q:=Q^{m_1} \times \cdots \times Q^{m_2}$. 

Given a $(\us,\uc)$-semistable complex $\cxE$ with Hilbert polynomials $P$ we can use the evaluation maps
\[ H^0(\cE^i(n)) \otimes \cO_X(-n) \ra \cE^i \]
along with a choice of isomorphism $V^i \cong H^0(\mathcal{E}^i(n))$ to parametrise the $i$th sheaf $\mathcal{E}^i$ by a point $q^i : V^i  \otimes \cO_X(-n) \ra \cE^i $ in $Q^i$. From the boundary morphisms $d^i : \cE^i \ra \cE^{i+1}$ we can construct a homomorphism
\[ \psi:=H^0(d(n)) \circ (\oplus_i H^0(q^i(n))) : \oplus_i V^i  \ra  \oplus_i H^0(\cE^i(n)) \]
where $d : \oplus_i \cE^i \rightarrow \oplus_i\cE^i$ is the morphism determined by the boundary maps $d^i$.
Such homomorphisms $\psi$ correspond to points in the fibres of the sheaf
\[  \cR:=(\oplus_i V^i)^\vee \otimes p_* \left(\cU  \otimes (\pi_X^{Q \times X})^* \cO_X(n)\right)  \]
over $Q$ where $p : Q \times X \ra Q$ is the projection and $\oplus_iV^i \otimes (\pi_X^{Q \times X})^*\cO_X(-n) \ra \cU$ is the quotient sheaf over $Q \times X$ given by taking the direct sum of the pullbacks of the universal quotients $V^i \otimes (\pi_X^{Q^i \times X})^* \cO_X(-n) \rightarrow \cU^i$ on $Q^i \times X$ to $Q \times X$. Note that $\cR$ is locally free for $n$ sufficiently large and so we can consider the projective bundle $\fD :=\PP(\cR \oplus \cO_Q)$ over $Q$. A point of $\fD$ over $q=(q^i:V^i \otimes \cO_X(-n) \ra \cE^i)_i \in Q$ is given by a pair $(\psi: \oplus_i V^i  \ra  \oplus_i H^0(\cE^i(n)),\zeta \in \CC)$ defined up to scalar multiplication. The parameter scheme $\fT$ consists of points $(q,[\psi : \zeta])$ in $\fD$ such that:
\begin{enumerate}
\renewcommand{\labelenumi}{\roman{enumi})}
\item $\psi =H^0(d(n)) \circ (\oplus_i H^0(q^i(n)))$ where $d : \oplus_i \cE^i \ra \oplus_i\cE^i$ is given by morphisms $d^i : \cE^i \ra \cE^{i+1}$ which satisfy $d^i \circ d^{i-1} = 0$,
\item $ \zeta\neq 0$.
\end{enumerate}
The conditions given in i) are all closed (they are cut out by the vanishing locus of homomorphisms of locally free sheaves) and condition ii) is open; therefore $\fT $ is a locally closed subscheme of $\fD$. We let $\fD'$ denote the closed subscheme of $\fD$ given by points which satisfy condition i). We will write points of $\fT$ as $(q,d)$ where $q=(q^i:V^i \otimes \cO_X(-n) \ra \cE^i)_i \in Q$ and $d$ is given by $d^{i} : \cE^i \rightarrow \cE^{i+1}$ which satisfy $d^{i} \circ d^{i-1}=0$.

\begin{rmk} The construction of the parameter scheme $\fT$ depends on the choice of $n$ and the Hilbert polynomials $P$. We  write $\fT_{P}$ or $\fT(n)$ if we wish to emphasise its dependence on $P$ or $n$.
\end{rmk}

\subsection{The group action}\label{gp act}

For $m_1 \leq i \leq m_2$ we have fixed vector spaces $V^i$ of dimension ${P^i(n)}$. The reductive group $\Pi_i \GL(V^i)$ acts on both $Q$ and $\fD$: if $g = (g_{m_1}, \dots, g_{m_2}) \in \Pi_i \GL(V^i)$ and $z = ((q^i : V^i \otimes \cO_X(-n) \ra \cE^i)_i,[\psi : \zeta]) \in \fD$, then
\[ g \cdot z = ((g_i \cdot q^i : V^i \otimes \cO_X(-n) \ra \cE^i)_i, [g \cdot \psi : \zeta]) \]
where
\[\xymatrix@1{
g_i \cdot q^i : & V^i \otimes \cO_X(-n) \ar[r]^{g_i^{-1 }\cdot} & V^i \otimes \cO_X(-n) \ar[r]^>>>>>{q^i} & \cE^i }\]
% \begin{center}{$ \begin{diagram}  \node{g_i \cdot q^i :V^i \otimes \cO_X(-n)} \arrow{e,t}{g_i^{-1 }\cdot} \node{V^i \otimes \cO_X(-n)} \arrow{e,t}{q^i} \node{\cE^i} \end{diagram} $}\end{center}
and
\[\xymatrix@1{
g \cdot \psi : & \oplus_i V^i \ar[r]^{g^{-1 }\cdot} & \oplus_i V^i  \ar[r]^>>>>>{q^i} & \oplus_i H^0(\cE^i(n))}.\]
%\begin{center}{$ \begin{diagram}  \node{g \cdot \psi :\oplus_iV^i  } \arrow{e,t}{\oplus_ig_i^{-1 }\cdot} \node{\oplus_i V^i } \arrow{e,t}{\psi} \node{\oplus_i H^0(\cE^i(n)).} \end{diagram} $}\end{center}
If instead we consider $\tilde{\psi}:= \oplus_i H^0(q^i(n))^{-1} \circ \psi : \oplus_iV^i \ra \oplus_iV^i$ then this action corresponds to conjugating $\tilde{\psi}$ by $g$; that is,
\[ g \circ \tilde{\psi} \circ g^{-1} = \widetilde{g \cdot \psi} .\]

This action preserves the parameter scheme $\fT$ and the orbits correspond to isomorphism classes of complexes. As the subgroup $\CC^* ( I_{V_{m_1}}, \dots ,I_{V_{m_2}})$ acts trivially  on $\fD$, we are really interested in the action of $(\Pi_i \GL(V^i))/ \CC^* $. Given integers $\us =(\sigma_{m_1} , \dots , \sigma_{m_2})$ we can define a character
\[ \begin{array}{cccc} \det_{\us} :  &\Pi_i \GL(V^i) & \ra & \CC^* \\ &(g_i) & \mapsto & \Pi_i \det g_i^{\sigma_i} \end{array}\]
and instead consider the action of the group $G=G_{\us}:= \ker \det_{\us}$ which maps with finite kernel onto $(\Pi_i \GL(V_i))/ \CC^* $. 

\subsection{The linearisation}\label{linearisation schmitt}

Schmitt uses the stability parameters $(\us,\uc):=(\us,\delta\ue)$ to determine a linearisation of the $G$-action on the parameter space $\fT$ in three steps. The first step is given by using the parameters $\us$ to construct a proper injective morphism from $\fD$ to another projective bundle $\fB_{\us}$ over $Q$. The parameters $\us$ are used to associate to each point $z = (q,[\psi : \zeta]) \in \fD$ a nonzero decoration 
\[ \varphi_{\us}(z) :  ( V_{\us}^{\otimes r_{\us}})^{\oplus 2} \otimes \cO_X(-r_{\us}n) \rightarrow \det \cE_{\us}  \]
(defined up to scalar multiplication) where  $r_{\us} = \sum_i \sigma_i r^i$ and $V_{\us}:= \oplus_i (V^{i})^{\oplus \sigma_i}$ and $\cE_{\us}:= \oplus_i (\cE^{i})^{\oplus \sigma_i}$. The fibre of $\fB_{\us}$ over $q \in Q$ parametrises such homomorphisms $\varphi_{\us}$ up to scalar multiplication and the morphism $\fD \ra \fB_{\us}$ is given by sending $z = (q,[\psi : \zeta]) \in \fD$ to $(q, [\varphi_{\us}(z)]) \in \fB_{\us}$. The group $G \cong\SL(V_{\us}) \cap \Pi_i \GL(V^i)$ acts on $\fB_{\us}$ by acting on $Q$ and $V_{\us}$ and $\fD \ra \fB_{\us}$ is equivariant with respect to this action.

The second step is given by constructing a projective embedding $\fB_{\us} \ra B_{\us}$. This embedding is essentially given by taking the projective embedding of each $Q^i$ used by Gieseker \cite{gieseker_sheaves}. Recall that Gieseker gave an embedding of $Q^i$ into a projective bundle $B_i$ over the components $R_i$ of the Picard scheme of $X$ which contain the determinant of a sheaf $\cE^i$ parametrised by $Q^i$. This embedding is given by sending a quotient sheaf $q^i : V^i \otimes \cO_X(-n) \ra \cE^i$ to a homomorphism $\wedge^{r^i} V^i \ra H^0(\det \cE^i(r^in))$ which represents a point in a projective bundle $B_i$ over $R_i$. The group $\SL(V^i)$ acts naturally on $B_i$ by acting on the vector space $\wedge^{r^i} V^i$ and the morphism $Q^i \ra B_i$ is equivariant with respect to this action. In a similar way Schmitt also constructs an equivariant morphism $\fB_{\us} \ra B'_{\us}$ where $B_{\us}'$ is a projective bundle over the product $\Pi_i R_i$. Let $B_{\us} =  B_{m_1} \times \cdots \times B_{m_2} \times B'_{\us} $; then the map $ \fB_{\us} \ra B_{\us}$ is equivariant, injective and proper morphism (cf. \cite{schmitt05} Section 2.1).

The final step is give by choosing a linearisation on $B_{\us}$ and pulling this back to the parameter scheme $\fT$ via  
\[ \fT \hookrightarrow \fD \hookrightarrow \fB_{\us} \hookrightarrow B_{\us} =  B_{m_1} \times \cdots \times B_{m_2} \times B'_{\us} .\]
The schemes $B_i$ and $B'_{\us}$ have natural ample linearisations given by $\cL_i:=\cO_{B_i}(1)$ and $\cL':=\cO_{B'_{\us}}(1)$. The linearisation on $B_{\us}$ is given by taking a weighted tensor product of these linearisations and twisting by a character $\rho$ of $G=G_{\us}$. The character $\rho : G \ra \CC^*$ is the character determined by the rational numbers
\[ c_i :=  \left[ \sigma_i \left( \frac{P_{\us}(n)}{r_{\us} \delta (n)} - 1\right) \left( \frac{r_{\us}}{P_{\us}(n)} -  \frac{r^i}{P^i(n)}\right) - \frac{ r^i \eta_i}{P^i(n)} \right] \]
where $P_{\us}:= \sum_i \sigma_i P^i$;
that is, if these are integral we define
\[ \rho (g_{m_1},\cdots,g_{m_2})= \Pi_{i=m_1}^{m_2} \det g_i^{c_i} \]
and if not we can scale everything by a positive integer so that they become integral. We assume $n$ is sufficiently large so that $a_i = \sigma_i ({P_{\us}(n)} - r_{\us} \delta (n))/ r_{\us} \delta (n) + \eta_i $ is positive; these positive rational numbers $\underline{a}=(a_{m_1}, \dots , a_{m_2}, 1)$ are used to define a very ample linearisation
\[ \cL_{\underline{a}}:=\bigotimes_i \cL_i^{\otimes a_i} \otimes \cL \]
on $B_{\us}$ (where again if the $a_i$ are not integral we scale everything so that this is the case). The linearisation $\cL=\cL(\us,\uc)$ on $\mathfrak{T}$ is equal to the pullback of the very ample linearisation $\cL_{\underline{a}}^{\rho}$ on $B_{\us}$ where $\cL_{\underline{a}}^{\rho}$ denotes the linearisation obtained by twisting $\cL_{\underline{a}}$ by the character $\rho$. Of course this can also be viewed as a linearisation on the schemes $\fD'$ and $\fD$ too.

\subsection{Jordan-H\"older filtrations and S-equivalence}\label{sequiv sect}

%Vicky - why is D' projective
The moduli space of ($\underline{\sigma}, \underline{\chi}$)-semistable complexes with invariants $P$ is constructed as an open subscheme of the projective GIT quotient
\[ \fD' /\!/_{\cL} G \]
given by the locus where $\zeta \neq 0$ (by definition $\fT$ is the open subscheme of $\fD'$ given by this condition). Recall that the GIT quotient is topologically the semistable set modulo S-equivalence where two orbits are S-equivalent if their orbit closures meet in the semistable locus. This notion can be expressed in terms of Jordan-H\"older filtrations as follows:

\begin{defn}\label{sequiv}
A Jordan--H\"{o}lder filtration of a $(\us,\uc)$-semistable complex $\cxE$ is a filtration by subcomplexes
\[ 0_\cdot= \cxE_{[0]} \subsetneqq \cxE_{[1]} \subsetneqq \cdots \subsetneqq \cxE_{[k]} = \cxE \] 
such that the successive quotients $\cxE_{[i]}/ \cxE_{[i-1]}$ are $(\us,\uc)$-stable and
\[ P^\mathrm{red}_{\us,\uc}(\cxE_{[i]}/ \cxE_{[i-1]}) =P^\mathrm{red}_{\us,\uc}(\cxE) .\]
This filtration is in general not canonical but the associated graded object 
\[ \mathrm{gr}_{ (\underline{\sigma}, \underline{\chi})}(\cxE) := \bigoplus_{j=1}^k \cxE_{[j]}/ \cxE_{[j-1]}\] 
is canonically associated to $\cxE$ up to isomorphism. We say two ($\underline{\sigma}, \underline{\chi}$)-semistable complexes are {S-equivalent} if their associated graded objects with respect to ($\underline{\sigma}, \underline{\chi}$) are isomorphic.
\end{defn}

Jordan--H\"{o}lder filtrations of ($\underline{\sigma}, \underline{\chi}$)-semistable complexes exist in exactly the same way as they do for semistable sheaves (for example, see \cite{gieseker_sheaves}).

\subsection{The moduli space}
We are now able to state one of the main results of \cite{schmitt05} for us: the existence of moduli spaces of $(\us,\uc)$-semistable complexes. Recall that there is a parameter scheme $\fT=\fT_{P}$ (which is the open subscheme of $\fD'$ cut out by the condition $\zeta \neq 0$) with an action by a reductive group $G=G_{\us}$ such that the orbits correspond to isomorphism classes of complexes and the stability parameters determine a linearisation $\cL$ of this action. The moduli space is given by taking the open subscheme of the projective GIT quotient  $\fD'/\!/_{\mathcal{L}} G$ given by $\zeta \neq 0$.

\begin{thm}(\cite{schmitt05}, p3)\label{schmitt theorem}
Let $X$ be a smooth complex manifold, $P$ be a collection of Hilbert polynomials of degree $\dim X$ and $(\us,\uc)$ be stability parameters. There is a quasi-projective coarse moduli space \[M^{(\underline{\sigma}, \underline{\chi})-ss}(X,{P})\] for S-equivalence classes of $(\us,\uc)$-semistable complexes over $X$ with Hilbert polynomials $P$.
\end{thm}

\subsection{The Hilbert-Mumford criterion}\label{calc HM for cx}

The Hilbert-Mumford criterion allows us to determine GIT semistable points by studying the actions of one-parameter subgroups (1-PSs); that is, nontrivial homomorphisms $\lambda : \CC^* \ra G$. In this section we give some results about the action of 1-PSs of $G=G_{\us}$ on the parameter space $\fT$ for complexes (see also \cite{schmitt05} Section 2.1). 

We firstly study the limit of a point in $z = (q,[\psi : 1]) \in \fT$ under the action of a 1-PS $\lambda : \CC^* \rightarrow G$. For this limit to exist we need to instead work with a projective completion $\overline{\fT}$ of $\fT$. We take a projective completion which is constructed as a closed subscheme of a projective bundle $\overline{\fD}$ over the projective scheme $\overline{Q} := \Pi_i \overline{Q}^i$ where $\overline{Q}^i$ is the closure of $Q^i$ in the relevant quot scheme. The points of $\overline{\fD}$ over $q=(q^i : V^i \otimes \cO_X(-n) \ra \cE^i)_i \in \overline{Q}$ are nonzero pairs $[ \psi : \zeta]$ defined up to scalar multiplication where $\psi : \oplus_i V^i \ra \oplus_i H^0(\mathcal{E}^i(n))$ and $\zeta \in \CC$. Then $\overline{\fT}$ is the subscheme of points $(q,[\psi : \zeta]) \in \overline{\fD}$ such that $\psi =H^0(d(n)) \circ (\oplus_i H^0(q^i(n)))$ where $d : \oplus_i \cE^i \ra \oplus_i \cE^i $ is given by $d^i : \cE^i \ra \cE^{i+1}$ which satisfy $d^i \circ d^{i-1} =0$. It is clear that the group action and linearisation $\mathcal{L}$ extend to this projective completion.

Recall that $G \cong \SL(V_{\us}) \cap \Pi_i \GL(V^i)$ where $V_{\us} = \oplus_i (V^i)^{\oplus \sigma_i}$ and so a 1-PS $\lambda : \CC^* \ra G$ is given by a collection of 1-PSs $\lambda_i : \CC^* \to \GL(V^i)$ which satisfy \[\Pi_i \mathrm{det} \lambda_i(t)^{\sigma_i}= 1.\] 
A 1-PS $\lambda_i$ of $\GL(V^i)$ gives rise to a weight space decomposition of $V^i = \oplus_{j=1}^s V^i_j$ indexed by a finite collection of integers $k_1 > \cdots > k_s$
where $V^i_j=\{ v \in V^i : \lambda_i(t) \cdot v =t^{k_j} v \}$. This gives a filtration
\[ 0 \subsetneq V^i_{(1)}\subsetneq \cdots \subsetneq V^i_{(s)} = V^i\]
where $V^i_{(j)} := V^i_1 \oplus \cdots \oplus V^i_j$ and if we take a basis of $V^i$ which is compatible with this filtration then 
\[ \lambda_i(t) = \left( \begin{array}{ccc}t^{k_1} I_{V^i_1} & & \\ & \ddots & \\ && t^{k_s}I_{V^i_s} \end{array}\right) \] is diagonal. We can diagonalise each of these 1-PSs $\lambda_i$ simultaneously so there is a decreasing sequence of integers $k_1 > \cdots > k_s$ and for each $i$ we have a decomposition $V^i = \oplus_{j=1}^s V^i_j$ (where we may have $V_j^i =0$), and a filtration
\[ 0 \subset V^i_{(1)} \subset V^i_{(2)} \subset \cdots \subset V^i_{(s)} = V^i \]
for which $\lambda_i$ is diagonal.

Let $z=(q, [\psi : 1])$ be a point in $ \fT$ where $q = (q^i : V^i \otimes \cO_X(-n) \ra \cE^i)_i \in Q$; then we can consider its limit
\[ \overline{z} := \lim_{t \to 0} \lambda(t) \cdot z \]
under the 1-PS $\lambda$. By \cite{huybrechts} Lemma 4.4.3,
\[\overline{q}^i:= \lim_{t \to 0} \lambda_i(t) \cdot q^i = \oplus_{j=1}^s q_j^i : \oplus_{j=1}^s V_j^i \otimes \cO_X(-n) \rightarrow \oplus_{j=1}^s \mathcal{E}^i_j \]
where $\mathcal{E}^i_j$ are the successive quotients in the filtration 
\[ 0 \subset \mathcal{E}^i_{(1)} \subset \cdots  \subset \mathcal{E}^i_{(j)} := q^i (V^i_{(j)} \otimes \mathcal{O}_X(-n) ) \subset \cdots \subset \mathcal{E}^i_{(s)}= \mathcal{E}^i \]
induced by $\lambda_i$. For each $i$ we have a filtration of the corresponding sheaf $\cE^i$ induced by $\lambda_i$ and the boundary maps can either preserve this filtration or not. If they do then we say $\lambda$ induces a filtration of the point $z$ (or corresponding complex $\cxE$) by subcomplexes. It is easy to check the limit depends on whether $\lambda$ induces a filtration by subcomplexes or not:

\begin{lemma}\label{lemma on fixed pts}
Let $z=(q, [\psi : 1])$ be a point in $ \fT$ and $\lambda$ be a 1-PS of $G$ as above with weights $k_1> \cdots > k_s$. Then the limit
\[ \overline{z}:=\lim_{t \to 0} \lambda(t) \cdot z =(\overline{q}, [\overline{\psi}: \overline{\zeta}])\]
is given by $\overline{q}$ as above and $\overline{\psi} =H^0( \overline{d}(n)) \circ (\oplus_i H^0(\overline{q}^i(n)))$ where $\overline{d}$ is given by $\overline{d}^i : \oplus_j \cE^i_j \ra \oplus_j \cE^{i+1}_j$. Moreover:
\begin{enumerate}
\renewcommand{\labelenumi}{\roman{enumi})}
\item  If $\lambda$ induces a filtration by subcomplexes, then $\overline{\zeta} = 1$ and $\overline{d}^{i} = \oplus_{j=1}^s (d^{i}_j: \cE^i_j \ra \cE^{i+1}_j)$. In particular $\overline{z} \in \fT$ and the corresponding complex is the graded complex associated to the filtration induced by $\lambda$.
\item If $\lambda$ does not induces a filtration by subcomplexes let \[N := \min_{i,j,l} \{ k_l - k_j :  d^{i}(\mathcal{E}^i_{(j)}) \nsubseteq \mathcal{E}^{i+1}_{(l-1)} \} < 0.\] Then $\overline{\zeta} =0$ and we have $\overline{d}^{i}(\mathcal{E}^i_j) \cap \mathcal{E}^{i+1}_l = 0$ unless $k_l - k_j = N$. In particular, the limit $\overline{z}=(\overline{q},[\overline{\psi}:0])$ is not in the parameter scheme $\fT$.
\end{enumerate}
\end{lemma}

\begin{rmk}\label{rmk on fixed pts}
Let $z=(q, [\psi : \zeta])$ be a point in $\overline{\fT}$ given by $q=(q^i : V^i \otimes \cO_X(-n) \ra \cE^i)_i \in \overline{Q}$ and $\psi = H^0({d}(n)) \circ \oplus_i H^0({q}^i(n))$ where $d$ is defined by homomorphisms $d^i : \cE^i \ra \cE^{i+1}$. If $z$ is fixed by $\lambda$, then the quotient sheaves are direct sums $q^i = \oplus_j q^i_j : V^i \otimes \cO_X(-n) \ra \oplus_j \cE^i_j$ and so the boundary map $d^i$ can be written as $d^i = \oplus_{j,l} d^i_{l,j}$ where $d^i_{l,j}: \cE^i_j \ra \cE^i_l$. The fixed point locus of a 1-PS $\lambda : \CC^* \ra G$ acting on $\overline{\fT}$ decomposes into 3 pieces (each piece being a union of connected components) where these pieces are given by:
\begin{itemize}
 \item A diagonal piece consisting of points $z$ where $d^i=\oplus_j d^i_{j,j} $ is diagonal for all $i$ and $\zeta \in \CC$.
\item A strictly lower triangular piece consisting of points $z$ where $d^i = \oplus_{j<l} d^i_{l,j} $ is strictly lower triangular for all $i$ and $\zeta = 0$.
\item A strictly upper triangular piece consisting of points $z$ where $d^i = \oplus_{j>l} d^i_{l,j} $ is strictly lower triangular for all $i$ and $\zeta = 0$.
\end{itemize}
Note that by Lemma \ref{lemma on fixed pts} above, if we have a point $z \in \fT$ its limit under $\lambda(t)$ as $t \to 0$ is in either the diagonal or strictly lower triangular piece. In fact, we have $\lim_{t \to 0} \lambda(t) \cdot z \in \fT$ if and only if $\lambda$ induces a filtration of $z$ by subcomplexes.
\end{rmk}

Now we understand the limit points of $\CC^*$-actions we can compute the weight of the $\CC^*$-action on fixed points. By definition the Hilbert-Mumford function
\[ \mu^{\cL}(z,\lambda) = \mu^{\cL}(\lim_{t \to 0} \lambda(t) \cdot z , \lambda)\]
is equal to minus the weight of the $\lambda(\CC^*)$-action on the fibre of $\cL$ over $\overline{z}:=\lim_{t \to 0} \lambda(t) \cdot z$. By the construction of $\mathcal{L}$ we have 
\begin{equation}\label{form1 for HM}
\mu^\mathcal{L}({z}, \lambda)=\mu^{\mathcal{L}'}(\varphi_{\us}({z}), \lambda)+ \sum_{i=m_1}^{m_2} a_i \mu^{\mathcal{L}_i}({q}^i, \lambda_i) - \rho \cdot \lambda 
\end{equation}
where $\varphi_{\us}({z})$ is the decoration associated to ${z}$ and $a_i$ and $\rho$ are the rational numbers and character used to define $\cL$ (c.f. $\S$\ref{linearisation schmitt}). We let $P_{\us} = \sum_i \sigma_i P^i$ and $r_{\us} = \sum_i \sigma_i r^i$.

\begin{lemma}(\cite{schmitt05}, Section 2.1)\label{HM prop}
Let $\lambda$ be a 1-PS of $G$ which corresponds to integers $k_1> \cdots > k_s$ and decompositions $V^i = \oplus_{j=1}^s V^i_j$ for $m_1 \leq i \leq m_2$ as above. Let $z=(q,[\psi : 1]) \in \mathfrak{T}$ where $q=(q^i : V^i \otimes \mathcal{O}_X(-n) \rightarrow \mathcal{E}^i)_i \in Q$; then
\begin{enumerate}
\renewcommand{\labelenumi}{\roman{enumi})}
 \item If $\lambda$ induces a filtration of $z$ by subcomplexes
\[ {\mu^{\cL}(z, \lambda )} =  \sum_{i=m_1}^{m_2}  \sum_{j=1}^s k_j \left(  \sigma_i \frac{P_{\us}(n)}{r_{\us}\delta(n)} + \eta_i  \right) \rk \cE^i_j  \]
where $ \cE^i_j =\cE^i_{(j)} / \cE^i_{(j-1)}$ and $\cE^i_{(j)} = q^i( V^i_{(j)} \otimes \cO_X(-n))$.
\item If $\lambda$ does not induce a filtration of $z$ by subcomplexes
\[ {\mu^{\mathcal{L}}(z, \lambda )} =  \sum_{i=m_1}^{m_2} \sum_{j=1}^s k_j  \left(  \sigma_i \frac{P_{\us}(n)}{r_{\us}\delta(n)} + \eta_i  \right) \rk \mathcal{E}^i_j - N  \]
where $N$ is the negative integer given in Lemma \ref{lemma on fixed pts}.
\end{enumerate}
\end{lemma}
\begin{proof}
The weight of the action of $\lambda_i$ on $Q^i$ with respect to $\cL_i$ was calculated by Gieseker:
\[\mu^{\mathcal{L}_i}(q^i,\lambda_i) =\sum_{j=1}^s k_j \left(\rk\cE^i_j - \dim V^i_j \frac{r^i}{P^i(n)} \right) .\]
We can insert this into the formula (\ref{form1 for HM}) along with the exact values of $a_i$ and $c_i$ and use the fact that $\lambda$ is a 1-PS of $\SL(\oplus_i (V^i)^{\oplus \sigma_i})$ to reduce this to
\[ \mu^{\cL}(z, \lambda) =\mu^{\cL'}(\varphi_{\us}(z),\lambda) + \sum_{i=m_1}^{m_2} \sum_{j=1}^s k_j  \left(  \sigma_i \frac{P_{\us}(n)}{r_{\us}\delta(n)} -\sigma_i + \eta_i  \right) \rk \mathcal{E}^i_j . \]
Finally, by studying the construction of the decoration $\varphi_{\us}(z)$ associated to $z$ (for details see \cite{schmitt05}), we see that 
\[\mu^{\cL'}(\varphi_{\us}(z),\lambda) = \left\{ \begin{array}{ll}  \sum_{i=m_1}^{m_2} \sum_{j=1}^s k_j \sigma_i  \rk \cE^i_j & \mathrm{if} \: \lambda \mathrm{\:induces \: a \: filtration \: by \: subcomplexes} \\ \sum_{i=m_1}^{m_2} \sum_{j=1}^s k_j \sigma_i  \rk \cE^i_j - N & \mathrm{otherwise} \end{array} \right.\]
where $N$ is the negative integer of Lemma \ref{lemma on fixed pts}.
\end{proof}

\begin{rmk}\label{only need to worry about subcxs}
Schmitt observes that we can rescale the stability parameters by picking a sufficiently large integer $K$ and replacing $(\delta,\ue)$ with $(K \delta,\ue / K)$, so that for GIT semistability we need only worry about 1-PSs which induce filtrations by subcomplexes (cf. \cite{schmitt05}, Theorem 1.7.1). This explains why the test objects for (semi)stability in Definition \ref{stab defn} are subcomplexes rather than weighted sheaf filtrations.
\end{rmk}

\section{Stability conditions relating to cohomology}\label{sec on stab}

In this section we study these notions of stability for complexes in greater depth. As we are now studying complexes with varying invariants $P$ we do not impose any condition on $\ue$ (such as $\sum_i \eta_i r^i = 0$). An important property of these stability conditions is that we can describe any complex (of torsion free sheaves) as a finite sequence of extensions of semistable complexes by studying its Harder--Narasimhan filtration. 

In this section we describe a collection of stability conditions indexed by a small positive rational number $\epsilon$ which can be used to study the cohomology sheaves of a given complex. %We firstly study the limit as $\epsilon$ tends to zero, which is not a stability condition in the sense of Schmitt \cite{schmitt05}, however it is still possible to study the corresponding notion of stability. 
The stability parameters we are interested in are of the form $(\underline{1}, \delta \ue/\epsilon)$ where $\underline{1}$ is the constant vector and $\eta_i$ are strictly increasing rational numbers. For a given complex $\cxF$ with torsion free cohomology sheaves
\[ \cH^i(\cxF) := \ker d^i / \im d^{i-1}\] we show that the Harder--Narasimhan filtration of this complex encodes the Harder--Narasimhan filtration of the cohomology sheaves in this complex provided $\epsilon>0$ is sufficiently small.

\subsection{Harder--Narasimhan filtrations}

Given a choice of stability parameters $(\us, \uc)$ every complex has a unique maximal destabilising filtration known as its Harder--Narasimhan filtration:

%The notion of Harder--Narasimhan filtrations was first developed for vector bundles on curves in \cite{harder} where Harder and Narasimhan showed that any vector bundle has a filtration by subbundles such that the successive quotients are semistable with decreasing slopes. Such a filtration is unique and is called the Harder--Narasimhan filtration of the vector bundle. This notion can often be generalised to other objects and notions of stability.
 
%The stability conditions arising from stability parameters $ (\underline{\sigma}, \underline{\chi })$ give rise to Harder--Narasimhan filtrations of complexes of torsion free sheaves. A given complex has a Harder--Narasimhan filtration for each such notion of stability, however once we have fixed a set of stability parameters each complex has a unique Harder--Narasimhan filtration with respect to these parameters.
 
\begin{defn}\label{HN filtr}
Let $\cxF$ be a complex and $(\us,\uc)$ be stability parameters. A {Harder--Narasimhan filtration} for $\cxF$ with respect to $(\us,\uc)$ is a filtration by subcomplexes
\[ 0_\cdot= \cxF_{(0)} \subsetneqq \cxF_{(1)} \subsetneqq \cdots \subsetneqq \cxF_{(s)} = \cxF \]
such that the successive quotients $\cxF_j=\cxF_{(j)} / \cxF_{(j-1)}$ are complexes of torsion free sheaves which are $(\us,\uc)$-semistable and have decreasing reduced Hilbert polynomials with respect to $(\us,\uc)$:
\[ P_{\us, \uc}^{\mathrm{red}}(\cxF_{1}) > P_{\us, \uc}^{\mathrm{red}}(\cxF_2) > \cdots > P_{\us, \uc}^{\mathrm{red}}(\cxF_s) .\]
The Harder--Narasimhan type of $\cxF$ with respect to $(\us, \uc )$ is given by $\tau =({P}_{1}, \cdots {P}_{s})$ where ${P}_{j} = (P_j^i)_{i \in \ZZ}$ is the tuple of Hilbert polynomials of the complex $\cxF_j$ so that
\[ P^i_j := P(\cF^i_j)=P(\cF^i_{(j)}/\cF^i_{(j-1)}). \]
\end{defn} 
 
The Harder--Narasimhan filtration can be constructed inductively from the maximal destabilising subcomplex:
 
\begin{defn}
Let $\cxF$ be a complex and $(\us,\uc)$ be stability parameters. A subcomplex $\cxF_1 \subset \cxF$ is a {maximal destabilising subcomplex} for $\cxF$ with respect to $(\us,\uc)$ if
\begin{enumerate}
\renewcommand{\labelenumi}{\roman{enumi})}
\item The complex $\cxF_1$ is $(\us,\uc)$-semistable,
\item For every subcomplex $\cxE $ of $\cxF$ such that $\cxF_1 \subsetneq \cxE$ we have
\[  P_{\us,\uc}^{\mathrm{red}}(\cxF_1)  >  P_{\us, \uc}^{\mathrm{red}}(\cxE). \]
\end{enumerate}
\end{defn}

The existence and uniqueness of the maximal destabilising subcomplex follows in exactly the same way as the original proof for vector bundles of Harder and Narasimhan \cite{harder}.

\subsection{The limit as $\epsilon$ tends to zero }

Recall that we are interested in studying the collection of parameters $(\underline{1}, \delta \ue/\epsilon) $ indexed by a small positive rational number $\epsilon$ where $\underline{1}$ is the constant vector, $\eta_i$ are strictly increasing rational numbers and $\delta$ is a positive rational polynomial of degree $\max(\dim X -1 , 0)$. In this section we study the limit as $\epsilon$ tends to zero.

Observe that
\[ P_{\underline{1}, \delta \ue/\epsilon}^{\mathrm{red}}(\cxE) \leq P_{\underline{1}, \delta \ue/\epsilon}^{\mathrm{red}}(\cxF)  \]
is equivalent to 
\[  \epsilon \frac{\sum_i P(\cE^i)}{\sum_i \rk \cE^i} - \delta \frac{\sum_i \eta_i \rk \cE^i }{\sum_i \rk \cE^i} \leq \epsilon \frac{\sum_i P(\cF^i)}{\sum_i \rk \cF^i} - \delta \frac{\sum_i \eta_i \rk \cF^i }{\sum_i \rk \cF^i}\]
and if we take the limit as $\epsilon \to 0$ we get
\begin{equation}\label{eps is zero ineq} \frac{\sum_i \eta_i \rk \cE^i}{\sum_i \rk\cE^i} \geq \frac{\sum_i \eta_i \rk \cF^i}{\sum_i \rk\cF^i}. \end{equation}
We say $\cxF$ is $(\underline{0}, \delta \ue)$-semistable if all nonzero proper subcomplexes $\cxE \subset \cxF$ satisfy the inequality (\ref{eps is zero ineq}).
This is a slight generalisation of the parameters consider by Schmitt in \cite{schmitt05} where we now allow $\sigma_i$ to be zero. These generalised stability parameters will no longer define an ample linearisation on the parameter space (cf. $\S$\ref{linearisation schmitt}), but we can still study the corresponding notion of semistability. 
%We suppose $\underline{\chi} = \delta \underline{\eta_i}$ as usual and that $\eta_i < \eta_{i+1}$ for each $i$.

\begin{lemma}\label{ss for sigma0}
Suppose $\eta_i < \eta_{i+1}$ for all integers $i$; then the only $(\underline{0}, \delta \ue)$-semistable complexes (of torsion free sheaves) are shifts of (torsion free) sheaves and complexes which are isomorphic to a shift of the cone of the identity morphism of (torsion free) sheaves.
\end{lemma}
\begin{proof}
If $\cxF$ is a shift of a torsion free sheaf $\cF^k$, then a subcomplex $\cxE$ of $\cxF$ is just a subsheaf and it is trivial to verify it is $(\underline{0}, \delta \ue)$--semistable. If $\cxF$ is isomorphic to a shift of the cone on the identity morphism of a torsion free sheaf, then there is an integer $k$ such that $d^k$ is an isomorphism and $\cF^i= 0$ unless $i = k$ or $k+1$. A subcomplex $\cxE$ of $\cxF$ is either concentrated in position $k+1$ or concentrated in $[k,k+1]$. In the second case we must have that $\rk \cE^k \leq \rk \cE^{k+1}$ and in both cases it is easy to verify the inequality for $(\underline{0}, \delta \ue)$-semistability using the fact that the $\eta_i$ are strictly increasing.

Now suppose $\cxF$ is $(\underline{0}, \delta \ue)$-semistable. If all the boundary morphisms $d^i$ are zero, then each nonzero sheaf $\cF^k$ is both a subcomplex and quotient complex and so by semistability
\[ \eta_k = \frac{\sum_i \eta_i \rk \cF^i}{\sum_i \rk\cF^i}. \]
As the $\eta_i$ are strictly increasing, there can be at most one $k$ such that $\cF^k$ is nonzero. If there is a nonzero boundary map $d^k$ then the image of this boundary map can be viewed as a quotient complex (in position $k$) and a subcomplex (in position $k+1$) so that
\begin{equation}\label{eqn for eta 1} \eta_{k} \leq  \frac{\sum \eta_i \rk \cF^i}{\sum \rk\cF^i} \leq \eta_{k+1} .\end{equation}
As the $\eta_i$ are strictly increasing, there can be at most one $k$ such that $d^k$ is nonzero. From above, we see that $\cF^i = 0$ unless $i =k$ or $ k+1$. As the $\eta_i$ are increasing, we see that the inequalities of (\ref{eqn for eta 1}) must be strict. We can consider the kernel and cokernel of $d^k$ as a subcomplex and quotient complex respectively and by comparing the inequalities obtained from semistability with (\ref{eqn for eta 1}), we see that $d^k$ must be an isomorphism and so $\cxF$ is isomorphic to a shift of the cone on the identity morphism of $\cF^{k+1}$.
\end{proof}

\begin{lemma}\label{HN filtr for sigma0}
%Let $\delta$ be a positive rational polynomial of degree $d-1$ and let $\underline{\eta}$ be a tuple of rational numbers which are strictly increasing and define $\underline{\chi}=\delta \underline{\eta}$. 
Suppose $\eta_i < \eta_{i+1}$ for all integers $i$ and $\cxF $ is a complex. Let $k$ be the minimal integer for which $\cF^k$ is nonzero. Then the maximal destabilising subcomplex $\cxF_{(1)}$ of $\cxF $  with respect to $(\underline{0},\delta \ue)$ is 
\[ \cxF_{(1)} = \left\{ \begin{array}{cccccccccl} \d  \ra & 0 & \ra & \ker d^k & \ra & 0 &\ra & 0 & \ra \d & \quad \mathrm{if} \: \ker d^k \neq 0,  \\ \d  \ra & 0 & \ra &\cF^k & \ra & \im d^k &\ra  & 0 & \ra \d & \quad \mathrm{if} \: \ker d^k = 0.  \end{array} \right. \]
\end{lemma}
\begin{proof}
By Lemma \ref{ss for sigma0} these complexes are both $(\underline{0},\delta \ue)$-semistable. In order to prove this gives the maximal destabilising subcomplex we need to show that if $\cxF_{(1)} \subsetneq \cxE \subset \cxF$, then
\[ \frac{\sum_i \eta_i \rk \cE^i}{\sum  \rk \cE^i} > \frac{\sum_i \eta_i \rk \cF_{(1)}^i}{\sum  \rk \cF_{(1)}^i}. \]
As $\cxE \neq \cxF_{(1)}$, the set
\[ I:= \{ i \in \mathbb{Z} : \cE^i \neq \cF^i_{(1)} \}  \]
is nonempty. We note that if $i \in I$, then $\cE^i \neq 0 $.

Suppose $\ker d^k \neq 0$. If $k \in I$, then also $k+1 \in I$ as $\ker d^k \subsetneq \cE^k$ and so $0 \neq d(\cE^k) \subset \cE^{k+1}$. As the $\eta_i$ are strictly increasing we have $ \eta_{k} \rk \cE^k + \eta_{k+1} \rk \cE^{k+1} > \eta_k(\rk \cE^{k}  + \rk \cE^{k+1} ). $ If $i > k+1$ and belongs to $I$, then $ \eta_{i} \rk \cE^i  > \eta_k\rk \cE^{i} $. So
\[ \sum_{i \in I} \eta_i \rk \cE^i  > \sum_{i \in I} \eta_k \rk \cE^{i} \quad \mathrm{and} \quad  \sum_{i \notin I} \eta_i \rk \cE^i  = \sum_{i \notin I} \eta_k \rk \cE^{i}; \]
hence
\[  \frac{\sum \eta_i \rk \cE^i}{\sum  \rk \cE^i} >\eta_k = \frac{\sum \eta_i \rk \cF_{(1)}^i}{\sum  \rk \cF_{(1)}^i} .\]

The case when $\ker d^k = 0$ is proved in the same way.
%If $\ker d^k = 0$, then $k \notin I$ and so for all $i \in I$ we have $ \eta_i \rk \cE^i > {(\eta_k + \eta_{k+1})}\rk \cE^i/2 .$
%Since
%\[ \sum_{i \in I} \eta_i \rk \cE^i  > \sum_{i \in I}  \frac{\eta_k + \eta_{k+1}}{2} \rk \cE^{i}  \quad \mathrm{and} \quad \sum_{i \notin I} \eta_i \rk \cE^i  \geq \sum_{i \notin I}  \frac{\eta_k + \eta_{k+1}}{2} \rk \cE^{i},  \]
%we have
%\[  \frac{\sum \eta_i \rk \cE^i}{\sum  \rk \cE^i} >\frac{\eta_k + \eta_{k+1}}{2} = \frac{\sum \eta_i \rk \cF_{(1)}^i}{\sum  \rk \cF_{(1)}^i} .\]
\end{proof}

\begin{cor}\label{HNF corr coh rmk}
If $\cxF$ has torsion free cohomology sheaves, then its Harder--Narasimhan filtration with respect to these stability parameters picks out the kernels and images of the boundary maps successively:
\[\begin{array}{cccccccccc}
\cxF_{(1)} : \quad & \cdots \rightarrow 0 & \rightarrow & \ker d^k & \rightarrow & 0 & \rightarrow & 0 & \rightarrow & 0 \cdots \\
& & & \cap & & \cap & & \cap & & \\
\cxF_{(2)} : \quad & \cdots \rightarrow 0 & \rightarrow & \cF^k & \rightarrow & \im d^k & \rightarrow & 0 & \rightarrow & 0 \cdots \\
& & & \cap & & \cap & & \cap & & \\
\cxF_{(3)} : \quad & \cdots \rightarrow 0 & \rightarrow & \cF^k & \rightarrow & \ker d^{k+1} & \rightarrow & 0 \cdots & \rightarrow & 0 \\
& & & \cap & & \cap & & \cap & & \\
\cxF_{(4)} : \quad & \cdots \rightarrow 0 & \rightarrow & \cF^k & \rightarrow & \cF^{k+1} & \rightarrow & \im d^{k+1} & \rightarrow & 0 \cdots \\
& &  &  \vdots & & \vdots & & \vdots & &
\end{array} \]
In particular, the successive quotients are $\cH^i(\cxF)[-i]$ or isomorphic to $\mathrm{Cone}(\mathrm{Id}_{\im d^i})[-(i+1)] $.
\end{cor}

\subsection{Semistability with respect to $(\underline{1}, \delta \ue/\epsilon)$ }
Recall that a torsion free sheaf $\cF$ is semistable (in the sense of Gieseker) if for all subsheaves $\cE \subset \cF$ we have
\[ \frac{P(\cE)}{\rk \cE} \leq \frac{P(\cF)}{\rk \cF}. \]
A torsion free sheaf can be viewed as a complex (by placing it in any position $k$) and it is easy to see that $(\us, \uc)$-semistability of the associated complex is equivalent to (Gieseker) semistability of the sheaf. For $\epsilon$ a small positive rational number we consider the stability parameters $(\underline{1}, \delta \ue/\epsilon)$ where $\underline{\eta}$ are strictly increasing. 

\begin{lemma}\label{lemma X} Suppose $\cxF$ is a complex for which there is an $\epsilon_0>0$ such that for all positive rational $\epsilon < \epsilon_0$ this complex is $(\underline{1}, \delta \ue/\epsilon)$-semistable. Then $\cxF$ is either a shift of a Gieseker semistable torsion free sheaf or isomorphic to a shift of the cone on the identity morphism of a semistable torsion free sheaf.
\end{lemma}
\begin{proof}
By studying the limit as $\epsilon$ tends to zero we see that $\cxF$ is $(\underline{0}, \delta \ue)$-semistable and so $\cxF$ is either a shift of a torsion free sheaf or isomorphic to a shift of the cone on the identity morphism of a torsion free sheaf by Lemma \ref{ss for sigma0}. If $\cxF$ is the shift of a sheaf, then $(\underline{1}, \delta \ue/\epsilon)$-semistability for any $0<\epsilon < \epsilon_0 $ implies this sheaf must be Gieseker semistable. If $\cxF$ is the cone on the identity morphism of a torsion free sheaf $\cF$, then $\cF$ must be semistable as for any subsheaf $\mathcal{F}' \subset \mathcal{F}$ we can consider $\mathrm{Cone}(\mathrm{id}_{\cF'})$ as a subcomplex and so
\[  \frac{P(\cF')}{\rk \cF'} - \delta \frac{\eta_k + \eta_{k+1}}{2\epsilon} \leq   \frac{P(\mathcal{F})}{\rk \mathcal{F}} - \delta \frac{\eta_k + \eta_{k+1}}{2 \epsilon} \]
by $(\underline{1}, \delta \ue/\epsilon)$-semistability for any $0 < \epsilon < \epsilon_0$.
 \end{proof}
 
 \begin{rmk} \label{rmk on sigma0}
Conversely, a shift of a semistable torsion free sheaf or a shift of a cone on the identity morphism of a semistable torsion free sheaf is $(\underline{1}, \delta \ue/\epsilon)$-semistable for any $\epsilon >0$.
 \end{rmk}
 
\begin{rmk}
As $(\underline{1}, \delta \ue/\epsilon)$-semistability of a complex associated to a torsion free sheaf $\cF$ is equivalent to (Gieseker) semistability of $\cF$, it follows that the Harder--Narasimhan filtration of the associated complex with respect to $(\underline{1}, \delta \ue/\epsilon)$ is given by the Harder--Narasimhan filtration of the sheaf. Similarly we see that the Harder--Narasimhan filtration of $\mathrm{Cone}(\mathrm{id}_{\cF})$ with respect to $(\underline{1}, \delta \ue/\epsilon) $ is given by taking cones on the identity morphism of each term in the Harder--Narasimhan filtration of the sheaf $\cF$.
\end{rmk}

We have seen that the Harder--Narasimhan filtration of $\cxF$ with respect to $(\underline{0}, \delta \ue)$ picks out the successive kernels and images of each boundary map. In particular, the successive quotients are either of the form $\cH^i(\cxF)[-i]$ or isomorphic to $\mathrm{Cone}(\mathrm{Id}_{\im d^i})[-(i+1)] $.

\begin{thm}\label{HN filtrations for epsiloneta}
Let $\cxF$ be a complex concentrated in $[m_1,m_2]$ with torsion free cohomology sheaves. There is an $\epsilon_0 >0$ such that for all rational $0<\epsilon < \epsilon_0$ the Harder--Narasimhan filtration of $\cxF$ with respect to $(\underline{1}, \delta \ue/\epsilon)$ is given by refining the Harder--Narasimhan filtration of $\cxF$ with respect to $(\underline{0}, \delta \ue)$ by the Harder--Narasimhan filtrations of the cohomology sheaves $\cH^i(\cxF)$ and image sheaves $\im d^i$. % with respect to $(\underline{1}, \underline{\chi}/\epsilon)$. ]
\end{thm}
\begin{proof} Firstly, we note that if $\dim X = 0$ every sheaf is Gieseker semistable (all sheaves have the same reduced Hilbert polynomial) and so any choice of $\epsilon_0$ will work. Therefore, we assume $d = \dim X >0$. Let $\cH^i(\cxF)_j$ for $1 \leq j \leq s_i$ (resp. $\im  d^i_j$ for $1 \leq j \leq t_i$) denote the successive quotient sheaves in the Harder--Narasimhan filtration of $\cH^i(\cxF)$ (resp. $\im  d^i$). The successive quotients in this filtration are either shifts of $\cH^i(\cxF)_j$ or isomorphic to shifts of the cone on the identity morphism of $\im  d^i_j$ and so by Remark \ref{rmk on sigma0} these successive quotients are  $(\underline{1}, \delta \ue/\epsilon)$-semistable for any rational $\epsilon >0$. Thus it suffices to show there is an $\epsilon_0$ such that for all $0 <\epsilon <\epsilon_0$ we have inequalities
\[  \frac{P(\cH^{m_1}(\cxF)_1)}{\rk \cH^{m_1}(\cxF)_1} - \delta \frac{\eta_{m_1}}{\epsilon} >  \dots >  \frac{P(\cH^{m_1}(\cxF)_{s_{m_1}})}{\rk \cH^{m_1}(\cxF)_{s_{m_1}}} - \delta \frac{\eta_{m_1}}{\epsilon} >  \frac{P(\im d^{m_1}_1)}{\rk \im d_{m_1}^1} - \delta \frac{ \eta_{m_1} + \eta_{m_1 +1}}{2\epsilon} > \cdots \]
\[> \frac{P(\im d^{m_1}_{t_{m_1}})}{\rk \im d_{t_{m_1}}^{m_1}} - \delta \frac{ \eta_{m_1} + \eta_{m_1 +1}}{2\epsilon} >  \frac{P(\cH^{m_1}(\cxF)_1)}{\rk \cH^{m_1 +1}(\cxF)_1} - \delta \frac{\eta_{m_1 +1}}{\epsilon} > \cdots  > \frac{P(\cH^{m_2}(\cxF)_{s_{m_2}})}{\rk \cH^{m_2}(\cxF)_{s_{m_2}}} - \delta \frac{\eta_{m_2}}{\epsilon}. \]
Since we know that the reduced Hilbert polynomials of the successive quotients in the Harder--Narasimhan filtrations of the cohomology and image sheaves are decreasing, it suffices to show for $m_1 \leq i < m_2-1$ that:
\begin{equation*}\label{defining eps0} \begin{split} 
1) \: \: & \epsilon \frac{P(\cH^i(\cxF)_{s_i})}{\rk \cH^i(\cxF)_{s_i}} - \delta \eta_i > \epsilon \frac{P(\im d^i_1)}{\rk \im d^i_1} - \delta \frac{ \eta_i + \eta_{i+1}}{2} \: \: \mathrm{and}\\
2) \: \: & \epsilon \frac{P(\im d^i_{t_i})}{\rk \im d^i_{t_i}} - \delta \frac{ \eta_{i} + \eta_{i+1}}{2} >  \epsilon \frac{P(\cH^{i+1}(\cxF)_1)}{\rk \cH^{i+1}(\cxF)_1} - \delta \eta_{i+1}.
\end{split}
\end{equation*}
These polynomials all have the same top coefficient and we claim we can pick $\epsilon_0$ so if $0 < \epsilon < \epsilon_0$ we have strict inequalities in the second to top coefficients. Let $\mu(\mathcal{A})$ denote the second to top coefficient of the reduced Hilbert polynomial of $\mathcal{A}$ which is (up to multiplication by a positive constant) the slope of $\mathcal{A}$ and let $\delta^{\mathrm{top}} >0$ be the coefficient of $x^{d-1}$ in $\delta$. For $m_1 \leq i < m_2 -1$ let
\[ M_i := \max\left\{\mu(\im d^i_{1}) - \mu(\mathcal{H}^i(\cxF)_{s_i}), \mu(\mathcal{H}^{i+1}(\cxF)_{1}) - \mu(\im d^i_{t_i}) \right\}. \]
We pick $\epsilon_0 >0$ so that if $M_i >0$ then $ \epsilon_0 < \delta^{\mathrm{top}} ({\eta_{i+1} - \eta_i})/{2M_i} $.
\end{proof}

\section{The stratification of the parameter space}\label{sec on strat}

%In this section we study stratifications of the parameter space $\fT$ for complexes which come from the GIT set up and depend on the choice of stability parameters and the integer $n$ used to construct the parameter space $\fT =\fT(n)$. 

In the introduction we described a stratification of a projective $G$-scheme $B$ (with respect to an ample linearisation $\mathcal{L}$) by $G$-invariant subvarieties  $\{S_\beta  : \beta \in \mathcal{B} \}$ which is described in \cite{hesselink,kempf_ness,kirwan}. If we fix a compact maximal torus $T$ of $G$ and positive Weyl chamber $\mathfrak{t}_+$ in the Lie algebra $\mathfrak{t}$ of $T$, then the indices $\beta$ can be viewed as rational weights in $\mathfrak{t}_+$. %(recall that we can identify characters and cocharacters as we have fixed an inner product on the Lie algebra $\mathfrak{K}$ of the maximal compact subgroup $K$ of $G$). 
Associated to $\beta $ there is a parabolic subgroup $P_\beta \subset G$, a rational 1-PS $\lambda_\beta : \CC^* \ra T_{\CC}$ and a rational character $\chi_\beta : T_{\CC} \ra \CC^*$ which extends to a character of $P_\beta$.
%A map $\lambda : \CC^* \ra G$ (which is not necessarily a homomorphism) is a rational 1-PS if $\lambda_\beta( \CC^*)$ is a subgroup of $G$ and there is a integer $N$ such that $\lambda^N$ is a 1-PS of $G$.
By definition $Z_\beta$ is the components of the fixed point locus of $\lambda_\beta$ acting on $B$ on which $\lambda_\beta$ acts with weight $|| \beta ||^2$ and $Z_\beta^{ss}$ is the GIT semistable subscheme for the action of the reductive part $\Sb$ of $P_\beta$ on $Z_\beta$ with respect to the linearisation $\mathcal{L}^{\chi_{-\beta}}$. Then $Y_\beta$ (resp. $Y_\beta^{ss}$) is defined to be the subscheme of $B$ consisting of points whose limit under $\lambda_\beta(t)$ as $t \to 0$ lies in $Z_\beta$ (resp. $Z_\beta^{ss}$). %In particular, we have a retraction $p_\beta : Y_\beta \ra Z_\beta$. 
By \cite{kirwan} for $\beta \neq 0$ we have $S_\beta = G Y_\beta^{ss} \cong G \times^{P_\beta} Y_\beta^{ss}$.

In this section we study the stratifications of the parameter space $\fT_P(n)$ for complexes associated to the collection of stability conditions $(\underline{1}, \delta \ue/\epsilon)$ given in $\S$\ref{sec on stab}. We relate these stratifications to the natural stratification by Harder--Narasimhan types (see Theorem \ref{HN strat is strat} below).

\subsection{GIT set up}\label{GIT set up}

We consider the collection of stability parameters $(\underline{1}, \delta \ue/\epsilon)$ indexed by a small positive rational parameter $\epsilon$ where $\delta$ is a positive rational polynomial and $\ue$ are strictly increasing rational numbers. In $\S$\ref{sec on stab} we studied semistability and Harder--Narasimhan filtrations with respect to these parameters when $\epsilon$ is very small. The Harder--Narasimhan filtration of a complex with torsion free cohomology sheaves with respect to $(\underline{1}, \delta \ue/\epsilon)$ tells us about the Harder--Narasimhan filtration of the cohomology sheaves of this complex provided $\epsilon >0$ is chosen sufficient small (see Theorem \ref{HN filtrations for epsiloneta}). 

Recall that the parameter space $\fT = \fT_{P}(n)$ for $(\underline{1},\delta \ue /\epsilon)$-semistable complexes with invariants $P$ is a locally closed subscheme of a projective bundle $\fD$ over a product $Q=Q^{m_1} \times \cdots \times Q^{m_2}$ of open subschemes $Q^i$ of quot schemes. There is an action of a reductive group $G$ on $\fT$ where
\[ G = \SL(\oplus_i V^i) \cap \Pi_i \GL(V^i) \]
and $V^i$ are fixed vector spaces of dimension $P^i(n)$. The linearisation of this action is determined by the stability parameters (see $\S$\ref{linearisation schmitt} or \cite{schmitt05} for details). We also described a natural projective completion $\overline{\fT}$ of $\fT$ which is a closed subscheme of a projective bundle $\overline{\fD}$ over a projective scheme $\overline{Q}$ in $\S$\ref{calc HM for cx}. Note that the group $G$ and parameter scheme $\fT$ depend on the choice of a sufficiently large integer $n$. For any $n >\!>0$ and $\epsilon >0$, associated to this action we have a stratification of $\overline{\fT}$ into $G$-invariant locally closed subschemes such that the open stratum is the GIT semistable subscheme. As we are primarily interested in complexes with torsion free cohomology sheaves, which form an open subscheme ${\fT}^{tf}$ of the parameter space $\fT$, we look at restriction this stratification to the closure $\overline{\fT}^{tf}$ of ${\fT}^{tf}$ in $\overline{\fT}$:
\begin{equation}\label{snail} \overline{ \fT}^{tf} = \bigsqcup_{\beta \in \mathcal{B}} S_\beta. \end{equation}
Every point in $\fT^{tf}$ represents a complex which has a unique Harder--Narasimhan type with respect to $(\underline{1}, \delta \ue/\epsilon)$ and so we can write
\[ \fT^{tf} = \bigsqcup_{\tau} R_\tau \]
where the union is over all Harder--Narasimhan types $\tau$.

%By fixing a basis of each $V^i$ we determine a compact maximal torus $T_i$ of $\mathrm{GL}(V^i)$ which consists of the matrices which are diagonal with respect to this basis. We choose the positive Weyl chamber
%\[ (\mathfrak{t}_i)_+ := \{ i \mathrm{diag}(a_1, \dots a_{P^i(n)}) : a_1 \geq \cdots \geq a_{\dim V_i}\} \]
%in $\mathfrak{t}_i=$Lie ${T_i}$. The maximal tori $T_i$ of $\mathrm{GL}(V^i)$ determine a maximal torus $T$ of $G$ and the choices $(\mathfrak{t}_i)_+$ of positive Weyl chambers determine a positive Weyl chamber $\mathfrak{t}_+$ in $\mathfrak{t}=$Lie$T$. The index set $\mathcal{B}$ for the stratification $\{ S_\beta : \beta \in \mathcal{B} \}$ is a finite set of points in $\mathfrak{t}_+$. 

Let us fix a complex $\cxF$ with torsion free cohomology sheaves and invariants $\uP$. We can assume we have picked $\epsilon$ sufficiently small as given by Theorem \ref{HN filtrations for epsiloneta}, so that the successive quotients appearing in its Harder--Narasimhan filtration with respect to $(\underline{1}, \delta \ue/\epsilon)$ are defined using the successive quotients in the Harder--Narasimhan filtrations of $\cH^i(\cxF)$ and $\im d^i$. Let $\tau$ be the Harder--Narasimhan type of $\cxF$ with respect to these parameters and we assume this is a nontrivial Harder--Narasimhan type (i.e. $\cxF$ is unstable with respect to $(\underline{1}, \delta \ue/\epsilon)$). 

Let $H_{i,j}$ (resp. $I_{i,j}$) denote the Hilbert polynomial of the $j$th successive quotient $\cH^i(\cxF)_j$ (resp. $\im d^i_j$) in the Harder--Narasimhan filtration of the sheaf $\cH^i(\cxF)$ (resp. $\im d^i$) for $1 \leq j \leq s_i$ (resp. $1 \leq j \leq t_i$). We also let ${H}_{i,j} =( H^{k}_{i,j})_{k \in \ZZ}$ and ${I}_{i,j} =(I^{k}_{i,j})_{k \in \ZZ}$ denote the collection of Hilbert polynomials given by 
\[ H^{k}_{i,j} = \left\{ \begin{array}{ll} H_{i,j} & \mathrm{if} \: k = i, \\ 0 & \mathrm{otherwise}, \end{array} \right. \quad \mathrm{and} \quad I^{k}_{i,j}= \left\{ \begin{array}{ll} I_{i,j} & \mathrm{if} \: k = i,i+1, \\ 0 & \mathrm{otherwise}. \end{array} \right.\]
Then the Harder--Narasimhan type  of $\cxF$ is given by
\[ \tau = ( {H}_{m_1,1}, \dots, {H}_{m_1,s_{m_1}}, {I}_{m_1,1}, \dots, {I}_{m_1,t_{m_1}}, {H}_{m_1+1,1}, \dots, {H}_{m_2,s_{m_2}} ) \]
which we will frequently abbreviate to $\tau = ({H}, {I})$ where ${H}=(H^{k}_{i,j})_{i,j,k \in \ZZ}$ and ${I}=(I^{k}_{i,j})_{i,j,k \in \ZZ}$.

\begin{ass}\label{ass on epsilon}
 We may also assume $\epsilon$ is sufficiently small so that the only $(\underline{1}, \delta \ue/\epsilon)$-semistable complexes with Hilbert polynomials ${I}_{i,j}$ are isomorphic to cones on the identity morphism of a torsion free semistable sheaf.
\end{ass}

\subsection{Boundedness}

We first give a general boundedness result for complexes of fixed Harder--Narasimhan type:

\begin{lemma}\label{cxs HN type bdd} 
The set of sheaves occurring in a complex of torsion free sheaves with Harder--Narasimhan type $({P_{1}}, \cdots ,{P_{s}})$ with respect to $(\underline{\sigma}, \underline{\chi})$ is bounded.
\end{lemma}
\begin{proof}
This follows from a result of Simpson (see \cite{simpson} Theorem 1.1) which states that a collection of torsion free sheaves on $X$ of fixed Hilbert polynomial is bounded if the slopes of their subsheaves are bounded above by a fixed constant. Recall that the slope of a sheaf is (up to multiplication by a positive constant) the second to top coefficient in its reduced Hilbert polynomial. Let $\cxE$ be a complex with this Harder--Narasimhan type; then for any subcomplex $\cxG$ of $\cxE$ we have
\[ P_{\underline{\sigma}, \underline{\chi}}^{\mathrm{red}}(\cxG) \leq P_{\underline{\sigma}, \underline{\chi}}^{\mathrm{red}}(\cxE_{1}) = \frac{\sum_i \sigma_i P^i_1 - \delta \sum_i \eta_i r^i_1}{\sum_i \sigma_i r^i_1}=:R\]
where $\cxE_1$ is the maximal destabilising subcomplex of $\cxE$ which has Hilbert polynomials specified by ${P}_1$. Suppose $\cG$ is a subsheaf of $\cE^i$ and consider the subcomplex
\[ \cxG : \quad 0 \to \cdots \to 0 \to \cG \to \mathcal{E}^{i+1} \to \cdots \to \mathcal{E}^{m_2}; \]
then we have an inequality of polynomials
%of $\mathcal{E}_\cdot$ by $\mathcal{F}_j =0$ for $j<i$, $\mathcal{F}_i = \mathcal{F}$ and $\mathcal{F}_j =\mathcal{E}_i$ for $j>i$. Then we have
\[ \frac{\sigma_i P(\cG)  - \delta \eta_i \rk \cG+\sum_{j >i} \sigma_j P^j_1 - \delta \sum_{j>i} \eta_j r^j_1}{\sigma_i  \rk \cG+ \sum_{j>i} \sigma_j r^j_1} \leq R. \]
The top coefficients agree and so we have an inequality
 \begin{equation*}  \frac{\deg \cG}{ \rk \cG}\leq   
 \frac{\delta^{\mathrm{top}}}{\sigma_i} \left(\eta_i + \frac{  \sum_{j>i} \eta_j r^j_1}{ \rk \cG}\right) - \frac{\sum_{j>i}\sigma_j d^j_1 }{\sigma_i \rk \cG}   + \left( 1 + \frac{\sum_{j>i}\sigma_j r^j_1}{\sigma_i \rk \cG}\right)  \left(\frac{ \sum_j \sigma_j d^j_1 - \delta^{\mathrm{top}} \sum_j \eta_j r^j_1}{\sum_j \sigma_j r^j_1} \right)  \end{equation*}
where $d^j_1$ is (up to multiplication by a positive constant) the second to top coefficient of $P^j_1$ and $\delta^\mathrm{top}$ is (up to multiplication by a positive constant) the leading coefficient of $\delta$. Since the rank of $\cG$ is bounded, it follows that the slope of subsheaves $\cG \subset \cE^i$ are bounded. 
\end{proof}

As the set of sheaves occurring in complexes of a given Harder--Narasimhan type are bounded, we can pick $n$ so that they are all $n$-regular and thus parametrised by $\fT(n)$.

\begin{cor}\label{how to pick n cx}
Let $({P}_{1}, \cdots ,{P}_{s})$ be a Harder--Narasimhan type with respect to $(\underline{\sigma}, \underline{\chi})$. Then we can choose $n$ sufficiently large so that for $1\leq i_{1} < \dots < i_k \leq s$ all the sheaves occurring in a complex of torsion free sheaves with Harder--Narasimhan type $({P_{i_1}}, \cdots ,{P_{i_k}})$ are $n$-regular.
\end{cor} 

\begin{ass}\label{assum on n}
Let $\tau=({H}, {I})$ be the Harder--Narasimhan type of the complex $\cxF$ we fixed in $\S$\ref{GIT set up}. We assume $n$ is sufficiently large so that the statement of Corollary \ref{how to pick n cx} holds for this Harder--Narasimhan type. In particular this means every complex $\cxE$ with Harder--Narasimhan type $\tau$ is parametrised by $\fT^{tf}$.
\end{ass}

\subsection{The associated index}

In this section we associate to the complex $\cxF$ with Harder--Narasimhan type $\tau$ a rational weight $\beta(\tau,n)$ which we will show later on is an index for an unstable stratum in the stratification %$\{ S_\beta : \beta \in \mathfrak{B} \}$ of $\overline{\mathfrak{T}}^{tf}$ 
defined at (\ref{snail}) when $n$ is sufficiently large.

Let $z=(q,[\psi : 1]) $ be a point in $\mathfrak{T}^{tf}$ which parametrises the fixed complex $\cxF$ with Harder--Narasimhan type $\tau$.
As mentioned in the introduction, the stratification can also be described by using Kempf's notion of adapted 1-PSs and so rather than searching for a rational weight $\beta$ we look for a (rational) 1-PS $\lambda_{\beta}$ which is adapted to $z$. By definition, a 1-PS $\lambda$ is adapted to $z$ if it minimises the (normalised) Hilbert--Mumford function 
\[ \mu^{\cL} (z, \lambda) = \min_{\lambda'} \frac{ \mu^{\cL}(z, \lambda')}{|| \lambda'||}\]
 and is therefore most responsible for the instability of $z$. It is natural to expect that $\lambda$ should induce a filtration of $\cxF$ which is most responsible for the instability of this complex; that is, its Harder--Narasimhan filtration. To distinguish between the cohomology and image parts of this Harder--Narasimhan filtration we write the Harder--Narasimhan filtration of $\cxF$ as
\[ 0 \subsetneq \mathcal{A}^\cdot_{m_1,(1)} \subsetneq \cdots \subsetneq \mathcal{A}^\cdot_{m_1,(s_{m_1})} \subsetneq \mathcal{B}^\cdot_{m_1,(1)} \cdots \mathcal{B}^\cdot_{m_1,(t_{m_1})} \subsetneq \mathcal{A}^\cdot_{m_1 +1,(1)} \subsetneq \cdots \subsetneq \mathcal{A}^\cdot_{m_2,(s_{m_2})}=\cxF \]
where the quotient $\cA^\cdot_{k,j}$ (resp. $\cB^\cdot_{k,j}$) of $\cA^{\cdot}_{k,(j)}$ (resp. $\cB^\cdot_{k,(j)}$) by its predecessor is isomorphic to $\cH^k(\cxF)_j[-k]$ (resp. $\mathrm{Cone}(\mathrm{id}_{\im d^k_j})[-(k+1)]$).
Such a filtration induces filtrations of the vector space $V^i$ for $m_1 \leq i \leq m_2$:
\begin{equation}\label{filtr of Vi} 0 \subset V^i_{m_1,(1)} \subset \cdots \subset V^i_{m_1,(s_{m_1})} \subset W^i_{m_1,(1)} \subset \cdots \subset W^i_{m_1,(t_{m_1})} \subset  \dots \subset V^i_{m_2,(s_{m_2})} = V^i \end{equation}
where \[V^i_{k,(j)} := H^0(q^i(n))^{-1}H^0(\mathcal{A}^i_{k,(j)}(n)) \quad \mathrm{and} \quad W^i_{k,(j)} := H^0(q^i(n))^{-1}H^0(\mathcal{B}^i_{k,(j)}(n)).\] Let $V^i_{k,j}$ (respectively $W^i_{k,j}$) denote the quotient of $V^i_{k,(j)}$ (respectively $W^i_{k,(j)}$) by its predecessor in this filtration. Note that by the construction of the Harder--Narasimhan filtration (see Theorem \ref{HN filtrations for epsiloneta}) we have that $V^i_{k,j} =0$ unless $k=i$ and $W^i_{k,j} =0$ unless $k = i,i-1$ and we also have an isomorphism $W^i_{i,j} \cong W^{i+1}_{i,j}$.

Given integers $a_{k,j}$ for $m_1\leq k \leq m_2$ and $1 \leq j \leq  s_k$ and integers $b_{k,j}$ for $m_1\leq k < m_2 -1$ and $1 \leq j \leq  t_k$ which satisfy
\begin{equation}\begin{split}\label{decr weights} i) & \quad a_{m_,1} > \dots > a_{m_1,s_{m_1}} > b_{m_1,1} > \dots > b_{m_1,t_{m_1}} > a_{m_1 +1,1} > \dots > a_{m_2, s_{m_2}} \\ ii) & \quad \sum_{i=m_1}^{m_2} \sum_{j=1}^{s_i} a_{i,j} \dim V^i_{i,j} + 2 \sum_{i=m_1}^{m_2-1} \sum_{j=1}^{t_i} b_{i,j} \dim W^i_{i,j} =0,  \end{split}\end{equation}
we can define a 1-PS $\lambda(\underline{a}, \underline{b})$ of $G$ as follows. Let  $V^i_k = \oplus_{j=1}^{s_k} V^i_{k,j}$ and $W^i_k = \oplus_{j=1}^{t_k} W^i_{k,j}$; then define 1-PSs $\lambda_{k}^{H,i} : \mathbb{C}^* \rightarrow \mathrm{GL}(V^i_k)$ and $\lambda_{k}^{I,i}  : \mathbb{C}^* \rightarrow \mathrm{GL}(W^i_k)$ by 
\[ \lambda^{H,i}_k (t)= \left( \begin{array}{ccc} t^{a_{k,1}}I_{V^i_{k,1}} & & \\  & \ddots & \\  & & t^{a_{k,s_k}}I_{V^i_{k,s_k}}  \end{array} \right) \quad  \lambda_{k}^{I,i} (t)= \left( \begin{array}{ccc} t^{b_{k,1}}I_{W^i_{k,1}} & & \\  & \ddots & \\  & & t^{b_{k,t_k}}I_{W^i_{k,t_k}}  \end{array} \right) \]
Then  $\lambda(\underline{a}, \underline{b}):= (\lambda_{m_1}, \dots, \lambda_{m_2})$ is given by
\begin{equation}\label{1ps def} \lambda_i(t): = \left( \begin{array}{ccc} \lambda_{i-1}^{I,i}(t)  & &  \\ & \lambda^{H,i}_i(t) & \\ & & \lambda^{I,i}_i(t)  \end{array} \right) \in \GL(V^i)=\GL(W^i_{i-1} \oplus V^i_i \oplus W^i_i). \end{equation}

For all pairs $(\underline{a}, \underline{b})$ the associated 1-PS $\lambda(\underline{a}, \underline{b})$ of $G$ induces the Harder--Narasimhan filtration of $\cxF$ and so by Proposition \ref{HM prop} 
\begin{equation*}\begin{split} {\mu^{\mathcal{L}}(z, \lambda(\underline{a}, \underline{b}) )} = & \sum_{i=m_1}^{m_2}  \sum_{j=1}^{s_i} a_{i,j} \left( \frac{P_{\underline{1}}(n)}{r_{\underline{1}}\delta(n)} + \frac{{\eta}'_i}{\epsilon} \right)  \rk \mathcal{H}^i(\mathcal{F}_\cdot)_j   \\  & + \sum_{i=m_1}^{m_2-1}  \sum_{j=1}^{t_i} b_{i,j} \left( 2\frac{P_{\underline{1}}(n)}{r_{\underline{1}}\delta(n)} + \frac{{\eta}_i' + {\eta}'_{i+1}}{\epsilon} \right)  \rk \im d^i_j  \end{split}\end{equation*}
where $P_{\underline{1}} = \sum_i P^i$ and $r_{\underline{1}} = \sum_i r^i$ and $(\underline{1}, \delta \ue'/\epsilon)$ are the stability parameters associated to $(\underline{1}, \delta \ue/\epsilon)$ which satisfy $\sum_i \eta_i' r^i = 0$ (cf. Remark \ref{normalise}).

We define
\[ a_{i,j} := \frac{1}{\delta(n)} -\left( \frac{P_{\underline{1}}(n)}{r_{\underline{1}} \delta(n)} + \frac{{\eta}'_i}{\epsilon} \right) \frac{\rk(H_{i,j})}{H_{i,j}(n)} \quad \mathrm{and} \quad  b_{i,j} := \frac{1}{\delta(n)} -\left( \frac{P_{\underline{1}}(n)}{r_{\underline{1}} \delta(n)} + \frac{{\eta}'_i+{\eta}'_{i+1}}{2\epsilon} \right) \frac{\rk(I_{i,j})}{I_{i,j}(n)} \]
where $\rk(H_{i,j})$ and $\rk(I_{i,j})$ are the ranks determined by the leading coefficients of the polynomials $H_{i,j}$ and $I_{i,j}$.
% These are decreasing as required. To see this write these as polynomials (i.e. get rid of n) then look at coeff of x^{2d-1} see this is exactly what we were looking at before when choosing epsilon_0!
The rational numbers $(\underline{a}, \underline{b})$ defined above are those which minimise the normalised Hilbert--Mumford function subject to condition $ii$) of (\ref{decr weights}) where $z \in \mathfrak{T}$ is the point which represents the complex $\cxF$ with Harder--Narasimhan type $\tau$. The choice of $\epsilon$ given by Theorem \ref{HN filtrations for epsiloneta} ensures that $(\underline{a}, \underline{b})$ also satisfy the inequalities $i$) of (\ref{decr weights}) for all sufficiently large $n$.

We choose a maximal torus $T_i$ of the maximal compact subgroup $\mathrm{U}(V^i)$ of $\mathrm{GL}(V^i)$ as follows. Take a basis of $V^i$ which is compatible with the filtration of $V^i$ defined at (\ref{filtr of Vi}) and define $T_i$ to be the maximal torus of $\mathrm{U}(V^i)$ given by taking diagonal matrices with respect to this basis. We %take the Borel subgroup of lower triangular matrices and this determines a 
pick the positive Weyl chamber 
\[ \mathfrak{t_i}_+ := \{ i \mathrm{diag}(a_1, \dots, a_{\dim V^i})  \in \mathfrak{t_i}: a_1 \geq \dots \geq a_{\dim V^i} \} .\]
%whose corresponding Borel subgroup is the group of lower triangular matrices.
%Then define $\beta_i =- i \mathrm{diag} (k_1, \dots, k_1, \dots, k_s, \dots k_s) \in \mathfrak{t_i}_+$ where $k_j$ appears $V_i^j$ times. The $\beta_i$ determine a point $\beta$ in the Lie algebra of $\tilde{G}(M)=\mathrm{SL}(M) \cap \prod \mathrm{GL}(V_i)$. 
Let $T$ be the maximal torus of the maximal compact subgroup of $G$ determined by the maximal tori $T_i$ and let $\mathfrak{t}_+$ be the positive Weyl chamber associated to the $\mathfrak{t_i}_+$.

\begin{defn}\label{defn of beta cxs}
%For a Harder--Narasimhan type $\tau$ with respect to $(\underline{\epsilon}, \delta \underline{\eta})$ 
We define $\beta = \beta(\tau,n) \in \mathfrak{t}_+$ to be the point defined by the rational weights
\[\beta_i = i \mathrm{diag} (b_{i-1,1}, \dots, b_{i-1,t_{i-1}}, a_{i,1} \dots, a_{i,s_i}, b_{i,1} \dots b_{i,t_i}) \in \mathfrak{t_i}_+\] where $a_{i,j}$ appears $H_{i,j}(n)$ times and $b_{k,j}$ appears $I_{k,j}(n)$ times. This rational weight defines a rational 1-PS $ \lambda_{\beta}$ of $G$ by $\lambda_{\beta} = \lambda(\underline{a}, \underline{b})$.
\end{defn}

\subsection{Describing components of $Z_\beta$}

%In this section we show that for $\epsilon$ sufficiently small and $n$ sufficiently large the rational weight $\beta=\beta(\tau)$ is an index for the stratification of $\overline{\mathfrak{T}}^{tf}$ and that $R_\tau$ is a union of connected components of $S_\beta \cap \mathfrak{T}$. 

By Remark \ref{rmk on fixed pts}, the $\lambda_{\beta}(\CC^*)$-fixed point locus of $\overline{\mathfrak{T}}^{tf}$ decomposes into three pieces: a diagonal piece, a strictly upper triangular piece and a strictly lower triangular piece. Each of these pieces decomposes further in terms of the Hilbert polynomials of the direct summands of each sheaf in this complex. We are interested in the component(s) of the diagonal part which may contain the graded object 
%\[ \bigoplus_{i=m_1}^{m_2} \bigoplus_{j=1}^{s_i} \mathcal{H}^i(\cxE)_j[-i] \oplus \bigoplus_{i=m_1}^{m_2-1} \bigoplus_{j=1}^{s_i} \mathrm{Cone}(\mathrm{id}_{\im d^i_j})[-(i+1)] \]
associated to the Harder--Narasimhan filtration of a complex $\cxE$ of Harder--Narasimhan type $\tau$.

We consider the closed subscheme $F_\tau$ of $\overline{\mathfrak{T}}^{tf}$ consisting of $z = (q, [\psi : \zeta])$ where as usual $\psi$ is determined by boundary maps $d^i$ and we have decompositions
\[ q^i = \bigoplus_{j=1}^{t_{i-1}} p^i_{i-1,j} \oplus \bigoplus_{j=1}^{s_i} q^i_{i,j} \oplus \bigoplus_{j=1}^{t_i} p^i_{i,j} \quad \mathrm{and} \quad d^i = \oplus_{j=1}^{t_i}d^i_j\]
where $q^i_{i,j} : V^i_{i,j} \otimes \mathcal{O}_X(-n) \rightarrow \mathcal{E}^i_{i,j}$ is a point in $\mathrm{Quot}(V^i_{i,j} \otimes \mathcal{O}_X(-n), H^i_{i,j})$ and $p^i_{k,j}: W^i_{k,j} \otimes \mathcal{O}_X(-n) \rightarrow \mathcal{G}^i_{k,j}$ is a point in $\mathrm{Quot}(W^i_{k,j} \otimes \mathcal{O}_X(-n), I^i_{k,j})$ and $d^i_j : \mathcal{G}^i_{i,j} \rightarrow \mathcal{G}_{i,j}^{i+1}$.

Following the discussion above we have:

% Vicky this is sort of confusing!
%Let $Q(H_i^{i,j})$ be the open subscheme of $\mathrm{Quot}(V_i^{i,j} \otimes \mathcal{O}_X(-n) , H_i^{i,j})$ consisting of torsion free quotient sheaves $q_i^{i,j}: V_i^{i,j} \otimes \mathcal{O}_X(-n) \rightarrow \mathcal{E}_i^{i,j}$ such that $H^0(q_i^{i,j}(n))$ is an isomorphism. Similarly let $Q(I_i^{k,j})$ be the open subscheme of $\mathrm{Quot}(W_i^{k,j} \otimes \mathcal{O}_X(-n) , I_i^{k,j})$ consisting of torsion free quotient sheaves $p_i^{k,j}: W_i^{k,j} \otimes \mathcal{O}_X(-n) \rightarrow \mathcal{G}_i^{k,j}$ such that $H^0(p_i^{k,j}(n))$ is an isomorphism. There is a direct sum map from the product
%\[ \prod_{i=1}^m \prod_{j=1}^{s_i} \overline{Q}(H_i^{i,j}) \times \prod_{i=1}^{m-1} \prod_{j=1}^{t_i}  \overline{Q}(I_i^{i,j}) \times  \overline{Q}(I_{i+1}^{i,j}) \]
%to $\overline{Q}$ and we let $\overline{Q}_{(\tau)}$ denote the image of this product in $\overline{Q}$.
%Then consider the subscheme 
%\[F_{(\tau)} = \left\{ \begin{array}{c} z=(q_1, \dots , q_m, [f, \epsilon ] \in \overline{\mathfrak{T}}^{tf} : \pi(z)=(q_i^{i,j}, p_i^{i,j}, p_{i+1}^{i,j}) \in \overline{Q}_{(\tau)} \\ \mathrm{and} \: f_{i,i+1} = \oplus_{j=1}^{t_i} f_{i,i+1}^j \mathrm{\: where \:} f_{i,i+1}^j : \mathcal{G}_i^{i,j} \rightarrow \mathcal{G}_{i+1}^{i,j} \end{array} \right\} \]
%of $\overline{\mathfrak{T}}^{tf}$ supported over $\overline{Q}_{(\tau)}$. By Lemma \ref{lemma on fixed pts} we have

\begin{lemma}\label{fixed pt locus}
$F_\tau$ is a union of connected components of the fixed point locus of $\lambda_{\beta}$ acting on $\overline{\mathfrak{T}}^{tf}$ which is contained completely in the diagonal part of this fixed point locus.
%Let $\tau$ be a Harder--Narasimhan type and $\beta = \beta(\tau)$. The fixed point locus of $\lambda_{-\beta}$ acting on the closure of the parameter scheme decomposes into diagonal, strictly upper triangular and strictly lower triangular components and each of these components decomposes further by looking at the Hilbert polynomials of the sheaves in these complexes. In particular, $F_{(\tau)}$ is a union of connected components of the fixed point locus.
\end{lemma}

\begin{rmk}\label{descr of F and Ttau for cx}
Every point in $\mathfrak{T}_{(\tau)}:=F_\tau \cap \mathfrak{T}^{tf}$ is a direct sum of complexes with Hilbert polynomials specified by $\tau$ and hence we have an isomorphism
\[ \mathfrak{T}_{{H}_{m_1,1}} \times \cdots \times  \mathfrak{T}_{{H}_{m_1,s_{m_1}}} \times \mathfrak{T}^{tf}_{{I}_{m_1,1}} \times \cdots  \times \mathfrak{T}^{tf}_{{I}_{m_1,t_{m_1}}} \times \mathfrak{T}_{{H}_{m_1 +1,1}}\times \cdots \times   \mathfrak{T}_{{H}_{m_2,s_{m_2}}} \cong  \mathfrak{T}_{(\tau)}.\]
\end{rmk}

Let $Z_\beta$ and $Y_\beta$ denote the subschemes of the stratum $S_\beta$ as defined in the introduction. 

%The scheme $Z_\beta$ is defined to be the set of points in $\overline{\mathfrak{T}}^{tf}$ which are fixed by $\lambda_{-\beta}$ and such that $\lambda_{-\beta}$ acts on these points with weight equal to $|| \beta||^2 $. 

\begin{lemma}\label{descr of Zbeta}
Let $ z \in \mathfrak{T}_{(\tau)}:=F_\tau \cap \mathfrak{T}^{tf} $; then $z \in Z_\beta$.
\end{lemma}
\begin{proof}
Let $z = ( q, [\psi: 1]) $ be a point of $\mathfrak{T}_{(\tau)}$ as above. The weight of the action of $\lambda_{\beta}$ on $z$ is given equal to $-\mu^{\mathcal{L}}(z, \lambda_\beta )$ and by Proposition \ref{HM prop} we have
\begin{equation*} {\mu^{\mathcal{L}}(z, \lambda_\beta )} =  \sum_{i=m_1}^{m_2}  \sum_{j=1}^{s_i} a_{i,j} \left( \frac{P_{\underline{1}}(n)}{r_{\underline{1}}\delta(n)} + \frac{{\eta}'_i}{\epsilon} \right)  \rk \cE^i_{i,j}    + \sum_{i=m_1}^{m_2-1}  \sum_{j=1}^{t_i} b_{i,j} \left( 2\frac{P_{\underline{1}}(n)}{r_{\underline{1}}\delta(n)} + \frac{{\eta}'_i + {\eta}'_{i+1}}{\epsilon} \right)  \rk \cG^i_{i,j}.  \end{equation*}
By definition $Z_\beta$ is the union of the connected components of the fixed point locus for the action of $\lambda_{\beta}$ on which $\lambda_{\beta}$ acts with weight $|| \beta||^2$ and it is easy to check that the choice of rational numbers $(\underline{a}, \underline{b})$ ensures that $|| \beta ||^2=-\mu^{\mathcal{L}}(z, \lambda_\beta )$ which completes the proof.
\end{proof}

\begin{cor}
The scheme $ \mathfrak{T}_{(\tau)} $ is a union of connected components of $Z_\beta \cap \mathfrak{T}^{tf}$.
\end{cor}

Let $F$ be the union of connected components of $Z_\beta$ meeting $\mathfrak{T}_{(\tau)}$; then $F $ is a closed subscheme of $\overline{\mathfrak{T}}^{tf}$ which is completely contained in the diagonal part of $Z_{\beta}$. Consider the subgroup
\[ \Sb = \left( \prod_{i=m_1}^{m_2} \prod_{j=1}^{s_i} \mathrm{GL}(V^i_{i,j}) \times \prod_{i=m_1}^{m_2-1} \prod_{j=1}^{t_i} \mathrm{GL}(W^i_{i,j}) \times \mathrm{GL}(W^{i+1}_{i,j}) \right) \cap \mathrm{SL}(\oplus_{i=m_1}^{m_2} V^i) \]
of $G$ which is the stabiliser of $\beta$ under the adjoint action of $G$. % where $V_i^{i,j}$ has dimension $H^{i,j}(n)$ and $W_i^{i,j}$ and $W_{i+1}^{i,j}$ have dimension $ I^{i,j}(n)$. 

\begin{lemma}\label{descr of Ghat}
$ \Sb$ has a central subgroup 
\[ \hat{G}= \left\{  (u_{i,j}, w_{i,j}) \in (\mathbb{C}^*)^{\sum_{i=m_1}^{m_2} s_i +\sum_{i=m_1}^{m_2-1} t_i} :  \prod_{i=m_1}^{m_2} \prod_{j=1}^{s_i} (u_{i,j})^{H_{i,j}(n)} \times \prod_{i=m_1}^{m_2-1} \prod_{j=1}^{t_i} (w_{i,j})^{2I_{i,j}(n)} =1  \right\} \]
which fixes every point of $F$. This subgroup acts on the fibre of $\mathcal{L}$ over any point of $F$ by a character ${\chi}_F : \hat{G} \rightarrow \mathbb{C}^* $ given by \[ {\chi}_F(u_{i,j},w_{i,j})=\prod_{i=m_1}^{m_2} \prod_{j=1}^{s_i} (u_{i,j})^{- \rk H_{i,j} \left(\frac{P_{\underline{1}}(n)}{r_{\underline{1}} \delta (n)} + \frac{\eta'_i}{\epsilon} \right) } \prod_{i=m_1}^{m_2-1} \prod_{j=1}^{t_i} (w_{i,j})^{- \rk I_{i,j} \left(\frac{2P_{\underline{1}}(n)}{r_{\underline{1}} \delta (n)} + \frac{\eta'_i + \eta'_{i+1}}{\epsilon} \right) }.\]
\end{lemma} 
\begin{proof}
The inclusion of $\hat{G}$ into $ \Sb$ is given by 
\[(u_{i,j} \: , \: w_{i,j}) \mapsto (u_{i,j} I_{V^i_{i,j}} \: , \: w_{i,j} I_{W^i_{i,j}} \: , \: w_{i,j} I_{W^{i+1}_{i,j}}) .\] 
Let $z = (q, [ \psi : \zeta]) $ be a point of $F$; then we have a decomposition
\[ q^i = \bigoplus_{j=1}^{t_{i-1}} p^i_{i-1,j} \oplus \bigoplus_{j=1}^{s_i} q^i_{i,j} \oplus \bigoplus_{j=1}^{t_i} p^i_{i,j} . \]
%where $q_i^{i,j}$ (respectively $p_i^{i,j}$ and $ p_i^{i-1,j}$) are points in $\overline{Q}(H_i^{i,j})$ (respectively $\overline{Q}(I_i^{i,j})$ and $\overline{Q}(I_i^{i-1,j})$) representing quotient sheaves $\mathcal{E}_i^{i,j}$ (respectively $\mathcal{G}_i^{i,j}$ and $\mathcal{G}_i^{i-1,j}$). 
A copy of $\mathbb{C}^*$ acts trivially on each quot scheme and so the central subgroup $\hat{G}$ fixes this quotient sheaf. The boundary maps are of the form $d^{i} = \oplus_{j=1}^{t_i} d^{i}_j$ where $d^{i}_j : \mathcal{G}^i_{i,j} \rightarrow \mathcal{G}^{i+1}_{i,j}$. As $ (u_{i,j} \: , \: w_{i,j}) \in \hat{G}$ acts on both $\mathcal{G}^i_{i,j}$ and $\mathcal{G}^{i+1}_{i,j}$ by multiplication by $w_{i,j}$, the boundary maps are also fixed by the action of $\hat{G}$. 

To calculate the character ${\chi}_F : \hat{G} \rightarrow \mathbb{C}^* $ with which this torus acts we fix $(u_{i,j} \: , \: w_{i,j}) \in \hat{G}$ and calculate the weight of the action of this element on the fibre over a point $z \in F$ by modifying the calculations for $\mathbb{C}^*$-actions in $\S$\ref{calc HM for cx} to general torus actions.
\end{proof}

Let $\cL^{\chi_{-\beta}}$ denote the linearisation of the $\Sb$ action on $Z_\beta$ given by twisting the original linearisation $\mathcal{L}$ by the character $\chi_{-\beta}$ associated to $-\beta$. Recall that $Z_\beta^{ss}$ is the GIT semistable set with respect to this linearisation. Consider the subgroup 
\begin{equation*} \begin{aligned}
G' &:= \prod_{i=m_1}^{m_2} \prod_{j=1}^{s_i} \SL(V^i_{i,j}) \times \prod_{i=m_1}^{m_2-1} \prod_{j=1}^{t_i} (\GL(W^i_{i,j}) \times \GL(W^{i+1}_{i,j}) )\cap \SL(W^i_{i,j} \oplus W^{i+1}_{i,j})
\\ &= \left\{ (g^i_j, h^i_{i,j}, h^{i+1}_{i,j} ) \in \Sb : \det g^i_j = 1 \: \mathrm{and} \: \det h^i_{i,j} \det h^{i+1}_{i,j} =1 \right\}. \end{aligned}\end{equation*}

\begin{prop}\label{descr of Fss}
Let $F$ be the components of $Z_\beta$ which meet $\mathfrak{T}_{(\tau)}$; then
 \[ F^{\Sb-ss}(\mathcal{L}^{\chi_{-\beta}}) = F^{G'-ss}(\mathcal{L}). \]
\end{prop}
\begin{proof}
There is a surjective homomorphism $\Phi $ from $ \Sb $ to the central subgroup $ \hat{G} $ defined in Lemma \ref{descr of Ghat}:
\[ \Phi( g_{j}^i, h^{i}_{i,j}, h^{i+1}_{i,j} ) = ( (\det g^{i}_j)^{D/H_{i,j}(n)} , (\det h^i_{i,j} \det h^{i+1}_{i,j})^{D/2I_{i,j}(n)}   ) \]
where $D= \Pi_{i=m_1}^{m_2} \Pi_{j=1}^{s_i} H_{i,j}(n) \times \Pi_{i=m_1}^{m_2-1} \Pi_{j=1}^{t_i} 2I_{i,j}(n)$. The composition of the inclusion of $\hat{G}$ into $\Sb$ with $\Phi$ is
\[ (u_{i,j}, w_{i,j}) \mapsto ( u_{i,j}^D, w_{i,j}^D). \] Therefore, $\ker  \Phi \times \hat{G}$ surjects onto $\Sb$ with finite kernel and since GIT semistability is unchanged by finite subgroups we have
$ F^{\Sb -ss} (\mathcal{L}^{\chi_{-\beta}}) = F^{\ker  \Phi \times \hat{G} -ss} (\mathcal{L}^{\chi_{-\beta}}) .$

%By definition the linearisation $\mathcal{L}^{\chi_{-\beta}}$ is equal to the linearisation $\mathcal{L}$ modified by the character corresponding to $-\beta$. If we 
Observe that the restriction of $\chi_\beta$ to the central subgroup $\hat{G}$
\[ \chi_\beta (u_{i,j},v_{i,j}) = \prod_{i=m_1}^{m_2} \prod_{j=1}^{s_i} u_{i,j}^{a_{i,j} H_{i,j}(n)} \times \prod_{i=m_1}^{m_2-1} \prod_{j=1}^{t_i} w_{i,j}^{b_{i,j} 2I_{i,j}(n)} \]
is equal to the character $\chi_F : \hat{G} \rightarrow \mathbb{C}^*$ defined in Lemma \ref{descr of Ghat}. As we are considering the action of $\ker  \Phi \times \hat{G}$ on $F$ linearised by $\mathcal{L}^{\chi_{-\beta}}$, the effects of the action of $\hat{G}$ and the modification by the character corresponding to $-\beta$ cancel so that
$  F^{\ker  \Phi \times \hat{G} -ss} (\mathcal{L}^{\chi_{-\beta}}) = F^{\ker  \Phi -ss} (\mathcal{L}) .$
Finally note that $G'$ injects into $\ker  \Phi$ with finite cokernel and so $F^{\ker  \Phi -ss} (\mathcal{L}) = F^{G' -ss} (\mathcal{L})$ which completes the proof.
%which proves the proposition.
\end{proof}
% Vicky : The character $\chi_\beta$ of Stab $\beta$ is given by projecting to the centre and a character on the centre. However, I'm not sure if this character acts trivially on Z(stab beta) / \hat{G}. Perhaps this doesn't matter as we know this character is trivial on G' which is what we end up with eventually.

\subsection{A description of the stratification}\label{sec with thm}

Consider the semistable subscheme \[\mathfrak{T}_{(\tau)}^{ss} := \mathfrak{T}_{(\tau)}^{\Sb -ss}(\mathcal{L}^{\chi_{-\beta}})\] for the $\Sb$-action on $ \mathfrak{T}_{(\tau)}$ with respect to $\mathcal{L}^{\chi_{-\beta}}$. % In this section for $n$ sufficiently large we describe this semistable subscheme (see Proposition \ref{descr of Ttauss}) and show the parameter space $R_{\tau}$ for complexes of this Harder--Narasimhan type is a union of connected components of $S_\beta \cap \fT^{tf}$ where $\beta = \beta(\tau,n)$ is the rational weight given in Definition \ref{defn of beta cxs} (see Theorem \ref{HN strat is strat}).
Recall from Remark \ref{descr of F and Ttau for cx} that we have an isomorphism
\[\mathfrak{T}_{(\tau)} \cong \mathfrak{T}_{{H}_{m_1,1}} \times \cdots \times \mathfrak{T}_{{H}_{m_1,s_{m_1}}} \times \mathfrak{T}^{tf}_{{I}_{m_1,1}} \times \cdots \times \mathfrak{T}^{tf}_{{I}_{m_1,t_{m_1}}} \times \mathfrak{T}_{{H}_{m_1 +1,1}}\times \cdots \times \mathfrak{T}_{{H}_{m_2,s_{m_2}}}.\]
Let 
\[z = \bigoplus_{i=m_1}^{m_2} \bigoplus_{j=1}^{s_i} z_{i,j}  \oplus \bigoplus_{i=m_1}^{m_2-1} \bigoplus_{j=1}^{t_i} y_{i,j} \]
be a point in $\mathfrak{T}_{(\tau)}$; that is, $z_{i,j}=(q^i_{i,j}, [0,1])$ is a point in $\mathfrak{T}_{{H}_{i,j}}$ corresponding to a complex $\mathcal{H}^\cdot_{i,j}$ concentrated in degree $i$ and $y_{i,j}=(p^i_{i,j}, p^{i+1}_{i,j} [\varphi^{i}_j,1])$ is a point in $\mathfrak{T}^{tf}_{{I}_{i,j}}$ corresponding to a complex $\mathcal{I}^\cdot_{i,j}$ concentrated in degrees $i$ and $i+1$. By Proposition \ref{descr of Fss}, we have \[\mathfrak{T}_{(\tau)}^{ss} =\mathfrak{T}_{(\tau)}^{G'-ss}(\mathcal{L}|_{\mathfrak{T}_{(\tau)}});\] therefore, $z$ is in $\mathfrak{T}_{(\tau)}^{ss}$ if and only if $\mu^{\mathcal{L}}(z,{\lambda}) \geq 0$ for every 1-PS $\lambda$ of $G'$. A 1-PS $\lambda$ of $G'$ is given by 
\begin{itemize} \item 1-PSs $\lambda^H_{i,j}$ of $\mathrm{SL}(V^i_{i,j})$ and 
 \item 1-PSs $\lambda^I_{i,j} =(\lambda^{I,i}_{i,j}, \lambda^{I,i+1}_{i,j})$ of $(\mathrm{GL}(W^i_{i,j}) \times \mathrm{GL}(W^{i+1}_{i,j}) )\cap \mathrm{SL}(W^i_{i,j} \oplus W^{i+1}_{i,j})$.
\end{itemize}

\begin{lemma}\label{lemma 1}
Suppose $n$ is sufficiently large. Then for any $z \in \mathfrak{T}_{(\tau)}$ as above for which a direct summand $\mathcal{H}^\cdot_{i,j}$ or $\mathcal{I}^\cdot_{i,j}$ is $(\underline{1}, \delta \ue /\epsilon)$-unstable, there is a 1-PS $\lambda$ of $G'$ such that $\mu^{\mathcal{L}}(z,{\lambda}) < 0$.
\end{lemma}
\begin{proof}
We suppose $n$ is sufficiently large so that Gieseker semistability of a torsion free sheaf with Hilbert polynomial $H_{i,j}$ (respectively $I_{i,j}$) is equivalent to GIT-semistability of a point in the relevant quot scheme representing this sheaf with respect to the linearisation given by Gieseker. We may also assume $n$ is sufficiently large so for $1 \leq i \leq m-1$ and $1 \leq j \leq t_i$, we have that $(\underline{1}, \delta \underline{\eta}/\epsilon)$-semistability of a complex with Hilbert polynomials ${I}_{i,j}$ is equivalent to GIT semistability of a point in $\mathfrak{T}_{{I}_{i,j}}$ for the linearisation defined by these stability parameters.

Firstly suppose $\mathcal{H}^{i}_{i,j}$ is unstable for some $i$ and $j$; then there exists a subsheaf $\mathcal{H}^{i,1}_{i,j} \subset \mathcal{H}^{i}_{i,j}$ such that
\begin{equation}\label{eqn H} \frac{H^0(\mathcal{H}^{i,1}_{i,j}(n))}{\rk \mathcal{H}^{i,1}_{i,j}} > \frac{H_{i,j}(n)}{\rk H_{i,j}} .\end{equation}
We construct a 1-PS $\lambda= (\lambda^H_{i,j}, \lambda^I_{i,j})$ of $G'$ with three weights $\gamma_1  > \gamma_2 = 0 > \gamma_3 $. Let \[V^{i,1}_{i,j}= H^0(q^i_{i,j}(n))^{-1}H^0(\mathcal{H}^{i,1}_{i,j}(n))\] and let $V^{i,3}_{i,j}$ be an orthogonal complement to $V^{i,1}_{i,j} \subset V^{i}_{i,j}$. Define \[\lambda^H_{i,j} = \left( \begin{array}{cc} t^{\gamma_1} I_{V^{i,1}_{i,j}} &  \\ & t^{\gamma_3} I_{V^{i,3}_{i,j}} \end{array} \right)\]
and define all the other parts of $\lambda$ to be trivial (the weights $\gamma_1$ and $\gamma_3$ should be chosen so $\lambda^H_{i,j}$ has determinant 1). Then by Proposition \ref{HM prop}
\[ \mu^{\mathcal{L}}(z,\lambda) = \left(\frac{ P_{\underline{1}}(n)}{r_{\underline{1}} \delta(n)} + \frac{\tilde{\eta}_i}{\epsilon}\right) \left[ H_{i,j}(n) \rk \mathcal{H}^{i,1}_{i,j} - H^0(\mathcal{H}^{i,1}_{i,j}(n)) \rk H_{i,j} \right] < 0 .\]

Secondly suppose $\mathcal{I}^\cdot_{i,j}$ is unstable with respect to $(\underline{1}, \delta\underline{\eta}/\epsilon)$; then by our assumption on $\epsilon$ it is not isomorphic to the cone on the identity map of a semistable sheaf. Let $d: \mathcal{I}^{i}_{i,j} \rightarrow \mathcal{I}_{i,j}^{i+1}$ denote the boundary morphism of this complex. If $d = 0$, then we can choose the 1-PS $\lambda$ to pick out the subcomplex $\cI_i^{i,j} \rightarrow 0$. For example, let $\lambda$ have three weights $1 > 0 > -1$ and let the only nontrivial part be \[\lambda^I_{i,j} = (tI_{W^i_{i,j}},t^{-1}I_{W^{i+1}_{i,j}});\] then $\mu^{\mathcal{L}}(z,\lambda) < 0 $. If $d \neq 0$ but has nonzero kernel, then consider the reduced Hilbert polynomial of this kernel. If the kernel has reduced Hilbert polynomial strictly larger than $\mathcal{I}^i_{i,j}$, then choose $\lambda$ to pick out the subcomplex $\ker d \rightarrow 0$. If the kernel has reduced Hilbert polynomial strictly smaller than $\mathcal{I}^i_{i,j}$, then choose $\lambda$ to pick out the subcomplex $0 \rightarrow \im  d$. If the kernel has reduced Hilbert polynomial equal to $I_{i,j}/ \rk I_{i,j}$, then choose $\lambda$ to pick out the subcomplex $\mathcal{I}^i_{i,j} \rightarrow \im  d$. In all three cases we see that $\mu^{\mathcal{L}}(z,\lambda) < 0 $. Finally, if $d$ is an isomorphism but $\mathcal{I}^i_{i,j}$ is not Gieseker semistable, then let $\mathcal{I}^{i,1}_{i,j}$ be its maximal destabilising subsheaf. A 1-PS which picks out the subcomplex $\mathcal{I}^{i,1}_{i,j} \rightarrow d^i_j(\mathcal{I}^{i,1}_{i,j})$ will destabilise $z$.
\end{proof}

\begin{lemma}\label{lemma 2}
Suppose $n$ is sufficiently large and let $z$ be a point in $\mathfrak{T}_{(\tau)}$ such that all the direct summands $\mathcal{H}^\cdot_{i,j}$ and $\mathcal{I}^\cdot_{i,j}$ are semistable with respect to $(\underline{1}, \delta \ue /\epsilon)$. If $\lambda$ is a 1-PS of $G'$ which induces a filtration of $z$ by subcomplexes, then $\mu^{\mathcal{L}}(z,{\lambda}) \geq 0$.
\end{lemma}
\begin{proof}
We suppose $n$ is chosen as in Lemma \ref{lemma 1}. The 1-PS $\lambda$ of $G'$ is given by 1-PSs $\lambda_{i,j}^H$ of $\mathrm{SL}(V^i_{i,j})$ and $\lambda^I_{i,j} =(\lambda^{I,i}_{i,j}, \lambda^{I,i+1}_{i,j})$ of $(\mathrm{GL}(W^i_{i,j}) \times \mathrm{GL}(W^{i+1}_{i,j}) )\cap \mathrm{SL}(W^i_{i,j} \oplus W^{i+1}_{i,j})$. We can diagonalise these 1-PSs simultaneously to get decreasing integers $\gamma_1 > \cdots > \gamma_u$  and decompositions $V^i_{i,j} = V^{i,1}_{i,j} \oplus \cdots \oplus V^{i,u}_{i,j}$ and similarly $W^i_{i,j} = W^{i,1}_{i,j} \oplus \cdots \oplus W^{i,u}_{i,j}$ and $W^{i+1}_{i,j}= W^{i+1,1}_{i,j} \oplus \cdots \oplus W^{i+1,u}_{i,j}$ such that
\[\lambda^H_{i,j}(t) = \left( \begin{array}{ccc} t^{\gamma_1} I_{V^{i,1}_{i,j}} & & \\ & \ddots & \\ & & t^{\gamma_u} I_{V^{i,u}_{i,j}} \end{array} \right)\]
and similarly for $\lambda^I_{i,j}$. The corresponding filtrations of these vector spaces give rise to filtrations of the sheaves $\mathcal{H}^i_{i,j}$, $\mathcal{I}^i_{i,j}$ and $\mathcal{I}^{i+1}_{i,j}$ and we let $\mathcal{H}^{i,k}_{i,j}$, $\mathcal{I}^{i,k}_{i,j}$ and $\mathcal{I}^{i+1,k}_{i,j}$ denote the successive quotients. As $\lambda$ induces a filtration by subcomplexes we have from Proposition \ref{HM prop} that
\begin{equation}\begin{split} \mu^{\mathcal{L}}(z,\lambda)  =   \sum_{i=m_1}^{m_2-1}  \sum_{j=1}^{t_i} \sum_{k=1}^u \gamma_k &\left[\left( \frac{ P_{\underline{1}}(n)}{r_{\underline{1}} \delta(n)} +\frac{{\eta}'_i}{\epsilon} \right) \rk \mathcal{I}^{i,k}_{i,j} + \left( \frac{ P_{\underline{1}}(n)}{r_{\underline{1}} \delta(n)} +\frac{{\eta}'_{i+1}}{\epsilon} \right) \rk \mathcal{I}^{i+1,k}_{i,j} \right] \\  &+  \sum_{i=m_1}^{m_2}  \sum_{j=1}^{s_i} \sum_{k=1}^u \gamma_k \left( \frac{ P_{\underline{1}}(n)}{r_{\underline{1}} \delta(n)} +\frac{{\eta}'_i}{\epsilon} \right) \rk \mathcal{H}^{i,k}_{i,j}. \end{split} \end{equation}
By construction of the linearisation (cf. the definition of $a_i$ in $\S$\ref{linearisation schmitt}), the numbers $P(n)/r\delta(n) + \eta_i'/\epsilon >0$. As $\mathcal{H}^i_{i,j}$, $\mathcal{I}^i_{i,j}$ and $\mathcal{I}^{i+1}_{i,j}$ are Gieseker semistable sheaves, 
\[  \sum_{k=1}^u \gamma_k \rk \mathcal{H}^{i,k}_{i,j} \geq 0 \quad \mathrm{and} \quad  \sum_{k=1}^u \gamma_k \left( \rk \mathcal{I}^{l,k}_{i,j} - \frac{\rk I_{i,j}}{I_{i,j}(n)} \dim  W^{l,k}_{i,j} \right) \geq 0 \quad \mathrm{for } \: l = i, i+1. \]
Therefore
\begin{equation*} \begin{split} \mu^{\mathcal{L}}(z,\lambda) & \geq  \frac{\rk I_{i,j}}{I_{i,j}(n)} \sum_{i=m_1}^{m_2-1}  \sum_{j=1}^{t_i} \sum_{k=1}^u \gamma_k \left[\left( \frac{ P_{\underline{1}}(n)}{r_{\underline{1}} \delta(n)} +\frac{{\eta}'_i}{\epsilon} \right)\dim  W^{i,k}_{i,j} + \left( \frac{ P_{\underline{1}}(n)}{r_{\underline{1}} \delta(n)} +\frac{{\eta}'_{i+1}}{\epsilon} \right) \dim  W^{i+1,k}_{i,j} \right] \\ & = \frac{\rk I_{i,j}}{I_{i,j}(n)} \sum_{i=m_1}^{m_2-1}  \sum_{j=1}^{t_i} \sum_{k=1}^u \gamma_k \left(\frac{{\eta}'_i}{\epsilon} \dim  W^{i,k}_{i,j} +  \frac{{\eta}'_{i+1}}{\epsilon}  \dim  W^{i+1,k}_{i,j} \right) \end{split} \end{equation*}
where the equality comes from the fact that $\lambda^I_{i,j}$ is a 1-PS of $ \mathrm{SL}(W^i_{i,j} \oplus W^{i+1}_{i,j})$ and so the weights satisfy 
$ \sum_{k=1}^u \gamma_k(\dim  W^{i,k}_{i,j} + \dim  {W}^{i+1,k}_{i,j}) =0.$
As $\lambda$ induces a filtration by subcomplexes, \[\dim  (W^{i,1}_{i,j} \oplus \dots \oplus W^{i,k}_{i,j}) \leq \dim  (W^{i+1,1}_{i,j} \oplus \dots \oplus W^{i+1,k}_{i,j})\]
 and it follows that $- \sum_{k=1}^u \gamma_k\dim  W^{i,k}_{i,j} =\sum_{k=1}^u \gamma_k\dim  W^{i+1,k}_{i,j} \geq 0$. Therefore
%Why? Answer: It is clear that we have:
 %\[ \sum_{k=1}^{u-1} (\gamma_k - \gamma_{k+1}) W^{i,[k]}_{i,j} \leq \sum_{k=1}^{u-1} (\gamma_k - \gamma_{k+1}) W^{i+1,[k]}_{i,j} \]
 %and so $\sum_{k=1}^u \gamma_k\dim  W^{i+1,k}_{i,j} \geq \sum_{k=1}^u \gamma_k\dim  W^{i,k}_{i,j} $
\[ \mu^{\mathcal{L}}(z,\lambda) \geq \frac{\rk I_{i,j}}{I_{i,j}(n)} \sum_{i=m_1}^{m_2-1}  \sum_{j=1}^{t_i}\frac{({\eta}'_{i+1} - {\eta}'_{i}) }{\epsilon}\sum_{k=1}^u \gamma_k \dim  W^{i+1,k}_{i,j} \geq 0 .\]
\end{proof}

Let $\mathfrak{T}^{ss}_{{H}_{i,j}} $ (resp. $\mathfrak{T}^{ss}_{{I}_{i,j}} $) be the subscheme of $\mathfrak{T}_{{H}_{i,j}} $ (resp. $\mathfrak{T}^{tf}_{{I}_{i,j}} $) which parametrises $(\underline{1}, \delta\ue /\epsilon)$-semistable complexes with Hilbert polynomials ${H}_{i,j}$ (resp. ${I}_{i,j}$).

\begin{prop}\label{descr of Ttauss}
%Let $\tau$ be a fixed Harder--Narasimhan type with respect to $(\underline{\epsilon}, \delta\underline{\eta})$ where $\underline{\eta}$ are strictly increasing rational numbers and let $\beta =\beta(\tau)$ be the associated rational weight. Then for $\epsilon >0 $ sufficiently small and 
For $n$ sufficiently large and by replacing $(\delta, \ue)$ by $(K\delta , \ue/K)$ for a sufficiently large integer $K$ we have an isomorphism
\[\mathfrak{T}_{(\tau)}^{ss} \cong \mathfrak{T}^{ss}_{{H}_{m_1,1}} \times \cdots \times \mathfrak{T}^{ss}_{{H}_{m_1,s_{m_1}}} \times \mathfrak{T}^{ss}_{{I}_{m_1,1}} \times \cdots \times \mathfrak{T}^{ss}_{{I}_{m_1,t_{m_1}}} \times \mathfrak{T}^{ss}_{{H}_{m_1 + 1,1}}\times \cdots \times \mathfrak{T}^{ss}_{{H}_{m_2,s_{m_2}}}.\]
\end{prop}
\begin{proof}
We suppose $n$ is chosen as in Lemma \ref{lemma 1} and let $z$ be a point in $\mathfrak{T}_{(\tau)}$. By Proposition \ref{descr of Fss}
\[ \mathfrak{T}_{(\tau)}^{ss} := \mathfrak{T}_{(\tau)}^{\Sb -ss}(\mathcal{L}^{\chi_{-\beta}}) = \mathfrak{T}_{(\tau)}^{G'-ss}(\mathcal{L}|_{\mathfrak{T}_{(\tau)}}) \]
and so $z$ is in $\mathfrak{T}_{(\tau)}^{ss}$ if and only if $\mu^{\mathcal{L}}(z,{\lambda}) \geq 0$ for every 1-PS $\lambda$ of $G'$.

If a direct summand $\mathcal{H}^\cdot_{i,j}$ or $\mathcal{I}^\cdot_{i,j}$ of $z$ is $(\underline{1}, \delta \ue /\epsilon)$-unstable, then $z \notin  \mathfrak{T}_{(\tau)}^{ss}$ by Lemma \ref{lemma 1}. By Lemma \ref{lemma 2} we have seen that if each of the direct summands $\mathcal{H}^\cdot_{i,j}$ and $\mathcal{I}^\cdot_{i,j}$ of $z$ are $(\underline{1}, \delta \ue /\epsilon)$-semistable and $\lambda$ induces a filtration by subcomplexes, then $\mu^{\mathcal{L}}(z,{\lambda}) \geq 0$. It follows from \cite{schmitt05} Theorem 1.7.1 (see also Remark \ref{only need to worry about subcxs}) that by rescaling $(\delta, \ue)$ to $(K\delta, \ue/K)$ for $K$ a large integer, we can verify GIT-semistability by only checking for 1-PSs which induce filtrations by subcomplexes. It follows that if each of the direct summands $\mathcal{H}^\cdot_{i,j}$ and $\mathcal{I}^\cdot_{i,j}$ of $z$ are $(\underline{1}, \delta \ue /\epsilon)$-semistable, then $z \in \fT_{(\tau)}^{ss}$. Therefore, $z \in \fT^{ss}_{(\tau)}$ if and only if all the direct summands $\mathcal{H}^\cdot_{i,j}$ and $\mathcal{I}^\cdot_{i,j}$ of $z$ are $(\underline{1}, \delta \ue /\epsilon)$-semistable. In particular, the above isomorphism comes from restricting the isomorphism given in Remark \ref{descr of F and Ttau for cx} to $\mathfrak{T}^{ss}_{(\tau)}$.
\end{proof}

Recall that there is a retraction $p_\beta : Y_\beta \rightarrow Z_\beta$ where
\[ p_\beta(y) = \lim_{t \to 0} \lambda_{\beta}(t) \cdot y. \]

% Vicky - I think we want $F^{ss}$ to be the connected components of $Z_\beta^ss$ meeting $\mathfrak{T}_{(\tau)}^ss$

\begin{lemma}\label{descr of pbeta inv of nice sch}
Let $F^{ss}$ denote the connected components of $Z_\beta^{ss}$ meeting $\mathfrak{T}_{(\tau)}^{ss}$; then for $n$ sufficiently large
\[ p_\beta^{-1}(F^{ss}) \cap \mathfrak{T}^{tf} = p_\beta^{-1}(\mathfrak{T}_{(\tau)}^{ss}). \]
\end{lemma}
\begin{proof}
Let $n$ be chosen as in Proposition \ref{descr of Ttauss}. Let $ y \in  p_\beta^{-1}(\mathfrak{T}_{(\tau)}^{ss})$ so that \[p_\beta (y)=\lim_{t \to 0} \lambda_{\beta}(t) \cdot y   \in  \mathfrak{T}^{ss}_{(\tau)} \subset F^{ss}.\] If $y \notin \mathfrak{T}^{tf}$, then for all $t \neq 0$ we have $\lambda_{\beta}(t) \cdot y \notin \mathfrak{T}^{tf}$ which would contradict the openness of $\mathfrak{T}^{tf} \cap F^{ss}$ in $F^{ss}$. 

Conversely suppose $y = (q^{m_1}, \dots , q^{m_2}, [\varphi :1]) \in \mathfrak{T}^{tf}$ and $z=p_\beta(y) \in F^{ss}$ where $q^i : V^i \otimes \mathcal{O}_X(-n) \rightarrow \mathcal{E}^i$ and $\varphi$ is given by $d^{i}: \mathcal{E}^i \rightarrow \mathcal{E}^{i+1}$. The scheme $F^{ss}$ is contained in the diagonal components of $Z_\beta^{ss}$; therefore the 1-PS $\lambda_\beta$ induces a filtration of $y$ by subcomplexes and the associated graded point is
\[z=(\oplus_{i,j} z_{i,j}) \oplus (\oplus_{i,j} y_{i,j} )\] where $ z_{i,j}=(q^i_{i,j},[0:1])$ and $y_{i,j} = (p^i_{i,j}, p^{i+1}_{i,j}, [d^{i}_j : 1])$ both represent complexes (cf. Lemma \ref{lemma on fixed pts}) and so $z = p_\beta (y) \in \fT_{(\tau)}$. By Proposition \ref{descr of Fss}, the limit $z$ is in the GIT semistable set for the action of $G'$ on $Z_\beta$ with respect to $\mathcal{L}$. We can apply the arguments used in the proof of Lemma \ref{lemma 1} to show that $z_{i,j} \in \mathfrak{T}^{ss}_{{H}_{i,j}} $ and $y_{i,j} \in \mathfrak{T}^{ss}_{{I}_{i,j}}$.
% Vicky - More careful explanation: We don't actually know if the complexes corr to z^{i,j} and y^{i,j} are complexes of torsion free sheaves. For z^{i,j}, the semistable set will be torsion free ss sheaves as in the Simpson case (if we had a cx (= a sheaf) in the closure of Q_i then it can be deformed to a t.free sheaf (see for example 4.4.8. Huybrechts and Lehn) and then we can show this t.free sheaf is semistable and in fact our original sheaf is iso to this sheaf.
%For y^{i,j} we have a cx of length two and its boundary map must be an iso for it to be GIT semistable (see proof of Proposition \ref{descr of Ttauss}). Say Y^{i,j} is the cone on some sheaf F then if F is not t.free again since its in the closure of Q_i we can deform it via \phi : F \to E where E is t.free. Then we can show Cone(E) is GIT semistable and in fact \phi is an isomorphism which proves F is t.free and semistable and H^0(-(n)) is an iso.
% This is enough to give the result.
\end{proof}

Recall that we have fixed Schmitt stability parameters $(\underline{1}, \delta \underline{\eta} / \epsilon)$ for complexes over $X$ where $\epsilon > 0$ is a rational number, $\eta_i$ are strictly increasing rational numbers indexed by the integers and $\delta$ is a positive rational polynomial such that $\deg \delta = \max (\dim  X -1,0)$. We may also assume that $\delta$ is sufficiently large (so that the scaling of Proposition \ref{descr of Ttauss} above has been done). We have assumed $\epsilon$ is very small and that $\tau$ is the Harder--Narasimhan type with respect to $(\underline{1}, \delta \underline{\eta} / \epsilon)$ of a complex $\cxF$ with torsion free cohomology sheaves and Hilbert polynomials $P=(P^{m_1}, \dots, P^{m_2})$. Let $\beta = \beta(\tau,n)$ be the rational weight given in Definition \ref{defn of beta cxs}. Provided $n$ is sufficiently large, all complexes with Harder--Narasimhan type $\tau$ may be represented by points in the scheme $\mathfrak{T}^{tf}=\fT^{tf}(n)$. We defined $R_\tau$ to be the set of points in $\mathfrak{T}^{tf}$ which parametrise complexes with this Harder--Narasimhan type $\tau$ and the following theorem provides $R_\tau$ with a scheme structure. There is an action of \[G = \Pi_{i=m_1}^{m_2} \mathrm{GL}(P^i(n)) \cap \mathrm{SL}(\Sigma_{i=m_1}^{m_2} P^i(n))\]
on this parameter scheme and the stability parameters determine a linearisation $\mathcal{L}$ of this action. Associated to this action there is a stratification $\{ S_\beta : \beta \in \mathcal{B} \}$ of the projective completion $\overline{\mathfrak{T}}^{tf}$ indexed by a finite set $\mathcal{B}$ of rational weights.

\begin{thm}\label{HN strat is strat}
For $n$ sufficiently large we have:
\begin{enumerate}
\renewcommand{\labelenumi}{\roman{enumi})}
\item $\beta = \beta(\tau,n)$ belongs to the index set $\mathcal{B}$ for the stratification $\{ S_\beta : \beta \in \mathcal{B} \}$ of $\overline{\mathfrak{T}}^{tf}$,
\item $R_\tau = G p_\beta^{-1}(\mathfrak{T}_{(\tau)}^{ss})$ and,
\item The subscheme $R_\tau  = G p_\beta^{-1}(\mathfrak{T}_{(\tau)}^{ss})$ of the parameter scheme $\mathfrak{T}^{tf}$ parametrising complexes with Harder--Narasimhan type $\tau$ is a union of connected components of $S_\beta \cap \mathfrak{T}^{tf}$.
\end{enumerate}
\end{thm}
\begin{proof}
Suppose $n$ is sufficiently large as in Proposition \ref{descr of Ttauss}. We defined $\beta$ by fixing a point $z = (q^{m_1}, \dots, q^{m_2} , [\varphi:1]) \in R_\tau$ corresponding to the complex $\mathcal{F}^\cdot$ with Harder--Narasimhan type $\tau$. We claim that $\overline{z} :=p_\beta(z) \in Z_\beta^{ss}$ which implies i). The 1-PS $\lambda_{\beta}$ induces the Harder--Narasimhan filtration of $\cxF$ and %by Corollary \ref{nice cor} we have that
 $\overline{z}= \lim_{t \to 0} \lambda_\beta(t) \cdot z $ is the graded object associated to this filtration. By Proposition \ref{descr of Ttauss} it suffices to show that each summand in the associated graded object is $(\underline{1}, \delta \underline{\eta} / \epsilon)$-semistable, but this follows by definition of the Harder--Narasimhan filtration.

In fact the above argument shows that $p_\beta^{-1}(\mathfrak{T}_{(\tau)}^{ss}) \subset R_\tau$ and since $R_\tau$ is $G$-invariant we have $G p_\beta^{-1}(\mathfrak{T}_{(\tau)}^{ss}) \subset R_\tau$. To show ii) suppose $y = (q^{m_1}, \dots, q^{m_2} ,[\varphi : 1]) \in R_\tau$ corresponds to a complex $\cxE$ with Harder--Narasimhan filtration 
\[ 0 \subsetneq \mathcal{H}^\cdot_{m_1,(1)} \subsetneq \cdots \subsetneq \mathcal{H}^\cdot_{m_1,(s_{m_1})} \subsetneq \mathcal{I}^\cdot_{m_1,(1)} \cdots \mathcal{I}^\cdot_{m_1,(t_{m_1})} \subsetneq \cdots \subsetneq \mathcal{H}^\cdot_{m_2,(s_{m_2})}=\mathcal{E}_\cdot \]
of type $\tau$. Then this filtration induces a filtration of each vector space $V^i$ and we can choose a change of basis matrix $g$ which switches this filtration with the filtration of $V^i$ given at (\ref{filtr of Vi}) used to define $\beta$. Then $g \cdot y  \in p_\beta^{-1}(\mathfrak{T}_{(\tau)}^{ss}) $ which completes the proof of ii). %This gives the set $R_\tau$ the structure of a scheme.

Since $F^{ss}$ is a union of connected components of $Z_\beta^{ss}$, the scheme $Gp_\beta^{-1}(F^{ss}) $ is a union of connected components of $S_\beta$. Therefore, $Gp_\beta^{-1}(F^{ss}) \cap \mathfrak{T}^{tf} $ is a union of connected components of $S_\beta \cap \mathfrak{T}^{tf}$. By ii) and Lemma \ref{descr of pbeta inv of nice sch}  
\[R_\tau = G p_\beta^{-1}(\mathfrak{T}_{(\tau)}^{ss})= Gp_\beta^{-1}(F^{ss}) \cap \mathfrak{T}^{tf} \] 
which proves iii).
\end{proof}

\section{Quotients of the Harder--Narasimhan strata}\label{sec on quot}

In the previous section we saw for $\epsilon$ very small and a fixed Harder--Narasimhan type $\tau$ with respect to $(\underline{1}, \delta \ue/\epsilon)$, there is a parameter space $R_\tau$ for complexes of this Harder--Narasimhan type and $R_\tau$ is a union of connected components of a stratum $S_\beta(\tau) \cap \fT^{tf}(n)$ when $n$ is sufficiently large. The action of $G$ on $\fT$ restricts to an action on $R_\tau$ such that the orbits correspond to isomorphism classes of complexes of Harder--Narasimhan type $\tau$. In this section we consider the problem of constructing a quotient of the $G$-action on this Harder--Narasimhan stratum $R_\tau$. If a suitable quotient did exist, then it would provide a moduli space for complexes of this Harder--Narasimhan type. In particular, it would have the desirable property that for two complexes to represent the same point it is necessary that their cohomology sheaves have the same Harder--Narasimhan type.

By \cite{hoskinskirwan} Proposition 3.6, any stratum in a stratification associated to a linearised $G$-action on a projective scheme $B$ has a categorical quotient. We can apply this to our situation and produce a categorical quotient of the $G$-action on $R_\tau$.

\begin{prop}
The categorical quotient of the $G$-action on $R_\tau$ is isomorphic to the product
\[ \prod_{i=m_1}^{m_2} \prod_{j=1}^{s_i}M^{(\underline{1}, \delta \ue/\epsilon)-ss}(X,{H}_{i,j}) \times \prod_{i=m_1}^{m_2 -1} \prod_{j=1}^{t_i}M^{(\underline{1}, \delta \ue/\epsilon)-ss}(X,{I}_{i,j}) \]
where $M^{(\underline{1}, \delta \ue/\epsilon)-ss}(X,{P})$ denotes the moduli space of $(\underline{1}, \delta \ue/\epsilon)$-semistable complexes with invariants $P$. Moreover:
\begin{enumerate}
\item A complex with invariants $H_{i,j}$ is just a shift of a sheaf and it is $(\underline{1}, \delta \ue/\epsilon)$-semistable if and only if the corresponding sheaf is Gieseker semistable.
\item A complex with invariants $I_{i,j}$ is concentrated in degrees $[i,i+1]$ and it is $(\underline{1}, \delta \ue/\epsilon)$-semistable if and only if it is isomorphic to a shift of the cone on the identity morphism of a Gieseker semistable sheaf.
\end{enumerate}
\end{prop}
\begin{proof}
It follows from \cite{hoskinskirwan} Proposition 3.6, that the categorical quotient is equal to the GIT quotient of $\Sb$ acting on $\mathfrak{T}_{(\tau)}$ with respect to the twisted linearisation $\cL^{\chi_{-\beta}}$. It follows from Proposition \ref{descr of Fss} this is the same as the GIT quotient of \[G'= \prod_{i=m_1}^{m_2} \prod_{j=1}^{s_i} \SL(V^i_{i,j}) \times \prod_{i=m_1}^{m_2-1} \prod_{j=1}^{t_i} (\GL(W^i_{i,j}) \times \GL(W^{i+1}_{i,j}) )\cap \SL(W^i_{i,j} \oplus W^{i+1}_{i,j})\] acting on $\mathfrak{T}_{(\tau)}$ with respect to $\cL$. By Theorem \ref{schmitt theorem}, this is the product of moduli spaces of $(\underline{1}, \delta \ue/\epsilon)$-semistable complexes with invariants given by $\tau$. The final statement follows from Lemma \ref{lemma X}, Remark \ref{rmk on sigma0} and the assumption on $\epsilon$ (cf. Assumption \ref{ass on epsilon}).
\end{proof}

In general this categorical quotient has lower dimension than expected and so is not a suitable quotient of the $G$-action on $R_\tau$. Instead, we suggest the quotient should be taken with respect to a perturbation of the linearisation used to provide the categorical quotient. However, as discussed in \cite{hoskinskirwan}, finding a way to perturb this linearisation and get an ample linearisation is not always possible. As $R_\tau = GY_{(\tau)}^{ss} \cong G \times^{P_\beta} Y_{(\tau)}^{ss}$ where $Y_{(\tau)}^{ss} :=p_\beta^{-1}(\mathfrak{T}_{(\tau)}^{ss})$, a categorical quotient of $G$ acting on $R_\tau$ is equivalent to a categorical quotient of $P_\beta$ acting on $Y_{(\tau)}^{ss}$. If we instead consider $P_\beta$ acting on $Y_{(\tau)}^{ss}$, then there are perturbed linearisations which are ample although $P_\beta$ is not reductive. A possible future direction is to follow the ideas of \cite{hoskinskirwan} and take a quotient of the reductive part $\Sb$ of $P_\beta$ acting on $Y_{(\tau)}^{ss}$ with respect to an ample perturbed linearisation and get a moduli space for complexes of Harder--Narasimhan type with $\tau$ some additional data.

\bibliographystyle{amsplain}
\bibliography{references}

\end{document}